\documentclass[lettersize,journal]{IEEEtran}
\usepackage{amsmath,amsfonts,bm}
\usepackage{amsthm}
\usepackage{array}
\usepackage[caption=false,font=normalsize,labelfont=sf,textfont=sf]{subfig}
\usepackage{multirow}
\usepackage{textcomp}
\usepackage{url}
\usepackage{verbatim}
\usepackage{graphicx}
\usepackage{xcolor}
\usepackage{colortbl}
\usepackage{algorithmicx,algpseudocode,threeparttable,booktabs}
\def\BibTeX{{\rm B\kern-.05em{\sc i\kern-.025em b}\kern-.08em
    T\kern-.1667em\lower.7ex\hbox{E}\kern-.125emX}}
\usepackage{balance}

\DeclareMathOperator*{\argmin}{arg\,min}
\newtheorem{theorem}{Theorem}
\newtheorem{remark}{Remark}
\begin{document}

\title{Efficient State Transition Algorithm With Guaranteed Optimality}

\author{Xiaojun Zhou, ~\IEEEmembership{Senior Member}, Chunhua Yang, ~\IEEEmembership{Fellow,~IEEE,}
	Weihua Gui, and Tingwen Huang, ~\IEEEmembership{Fellow,~IEEE}
	}



\maketitle

\begin{abstract}
The state transition algorithm (STA), as an intelligent optimization method grounded in constructivist learning, has been demonstrated to be highly effective in solving complex optimization problems. However, the standard STA suffers from slow convergence, particularly in the later stages when dealing with flat landscapes. Additionally, users are required to set the maximum number of iterations  based on intuition. To address these issues, an enhanced STA with guaranteed optimality is introduced. This improvement involves three key components. First, novel translation transformations, inspired by predictive modeling, are developed to generate a broader set of candidate solutions by leveraging historical data. Second, adaptive parameter control strategies are incorporated to accelerate convergence. Finally, a dedicated termination condition is designed to ensure that the algorithm converges at the optimal solution, analogous to the zero gradient condition in mathematical programming. The comprehensive experimental results validate the effectiveness and superiority of the proposed method.
The source codes for ESTA and EXSTA will be publicly available at: https://github.com/tiezhongyu2005/ESTA.
\end{abstract}


\section{Introduction}
\IEEEPARstart{M}{odern} systems engineering problems often involve complex optimization tasks spanning multiple stages of the system lifecycle, including design, analysis, and deployment. These problems are inherently characterized by nonlinearity, nonconvexity, high dimensionality, and uncertainty, thereby necessitating optimization methods that are both efficient and reliable to ensure satisfactory system performance. Intelligent optimization algorithms have been widely adopted as practical tools for solving complex optimization problems. However, most existing approaches, particularly those based on swarm intelligence and imitation learning, lack rigorous theoretical guarantees in terms of optimality, convergence, and controllability, thereby undermining their reliability in critical systems engineering applications.

The state transition algorithm (STA) is a constructivist learning-based intelligent optimization approach developed in recent years \cite{zhou2019statistical,du2025broad,dong2026learning}.
In STA, each solution to an optimization problem is metaphorically treated as a state, and its iterative update is correspondingly defined as a state transition.
Unlike the majority of swarm intelligence optimization algorithms based on imitation learning and behavioral cloning \cite{tang2021review}, the STA is inspired by constructivist learning \cite{supena2021influence}, which acquires the essential knowledge and constructs the necessary components, including
globality, optimality, rapidity, convergence and controllability, aiming to find a global or near-global optimal solution as quickly as possible.
In standard STA \cite{zhou2012state}, a set of diverse state transformation operators is specifically designed to generate candidate solutions to achieve this goal. Specifically, the expansion transformation facilitates global search to ensure globality, while the rotation transformation supports local search to enhance local optimality. Additionally, translation and axesion transformations are employed for heuristic search to improve rapidity. Furthermore, appropriate update strategies are implemented to ensure convergence. Moreover, each state transformation operator generates a regular neighborhood with an adjustable size, providing controllability. Due to its ease of understanding, convenience of use, as well as powerful search ability, it has been successfully applied in a wide range of practical engineering problems \cite{zhou2025interpretable, li2026hybrid, chen2026knowledge, zhou2026evolutionary}.

However, on the one hand,  the standard STA tends to exhibit slow convergence in the later stages, particularly when dealing with flat landscapes, due to the use of fixed parameter settings. Many efforts have been made to address this issue. In \cite{zhou2018dynamic}, all state transformation factors are decreased periodically to enhance local exploitation in the later stages. In \cite{zhou2019statistical}, an optimized strategy for parameter selection was proposed to accelerate convergence; however, it necessitates some time to initiate the acceleration process.
In \cite{dong2022adaptive}, gradient information is incorporated to accelerate the search process in the later stages; nevertheless, it requires the function to be differentiable.  The question of how to integrate appropriate mechanisms to accelerate convergence in STA for a general function requires further investigation.

On the other hand, the maximum number of iterations in standard STA is manually specified as a termination criterion, typically based on prior experience, which is common practice in intelligent optimization algorithms. As is well known, for an intelligent optimization algorithm, besides generating new candidates and updating incumbents, the termination condition is also a crucial component.
In mathematical programming with continuously differentiable objective function, a widely used termination condition is the zero gradient condition, as it guarantees the
optimality due to Fermat's theorem. However, for the majority of existing intelligent optimization algorithms, it is challenging to design an automatic termination criterion to guarantee the
optimality \cite{cai2024multiselection, fu2026hybrid, gao2026multiknowledge}. The simplest and most widely used termination criterion in intelligent optimization algorithms is a predefined number of function evaluations (or a maximum number of iterations). However, determining this value is highly problem-dependent and often relies on empirical tuning. This presents a tuning dilemma: an overly low value may fail to reach the optimum, whereas an excessively high one can lead to unnecessary computational overhead. What's more, using the maximum number of iterations as a stopping criterion can be detrimental, and how to set a good termination condition in intelligent optimization algorithms still remains unresolved and require further investigation \cite{ravber2022maximum}. Nevertheless, the issue of how to design an appropriate termination criterion to ensure optimality is often overlooked in studies on intelligent optimization, including previous research on STA. In earlier versions of STA, the termination criterion was also based on a predefined maximum number of iterations, which must be specified in advance and typically relies on expert knowledge. Notably, the intrinsic characteristics in STA, particularly the rotation transformation that enables the search within a hypersphere of a specified radius, make it possible to set a dedicated termination condition that can guarantee the optimality automatically.

To address the aforementioned issues, this study aims to develop a set of efficient STA variants that not only achieves faster convergence on flat landscapes but also guarantees optimality through an integrated automatic termination mechanism, which is a crucial yet often overlooked aspect in most existing studies on intelligent optimization algorithms. The main contributions and innovations presented in this paper are summarized as below.
\begin{enumerate}
  \item Novel translation transformations based on predictive modeling are developed to generate a larger pool of potential candidates and accelerate the convergence process.
  \item Parameter control strategies are incorporated into STA to promote faster convergence.
  \item Specific termination criteria are designed to provide theoretical guarantees of optimality.
\end{enumerate}

The rest of this paper is organized as follows. Section II reviews related work, including fundamental principles of STA and commonly used termination conditions in intelligent optimization. Section III provides details of the proposed STA, including the novel state transformations, parameter control strategies, and dedicated termination criteria. In Section IV, comprehensive experimental results and analyses are given to demonstrate the effectiveness and superiority of the proposed STA. Finally, Section V concludes the paper.

\section{Related Work}

In this study, the following general unconstrained optimization problem is considered:
\begin{eqnarray}
\min_{\bm x \in \Omega} f(\bm x),
\end{eqnarray}
where $f(\bm x)$ denotes a continuous function with a lower bound and $\Omega \subseteq \mathbb{R}^n$ is a set that is both closed and compact.

\subsection{Basics of STA}

The STA is an intelligent optimization method grounded in constructivist learning,  based on states and state transitions. Given the current state $\bm s_k$, the unified form for generating the candidate state $\bm s_{k+1}$ in STA is formulated as below:
\begin{equation}
\label{sta_framework}
\bm s_{k+1}= A_{k} \bm s_{k} + B_{k} \bm u_{k},
\end{equation}
where, $\bm s_{k}$ represents a state corresponding to a candidate solution of the optimization problem for the $k$th iteration; $\bm u_{k}$ can be treated either as a function of current and historical states or independently defined;
$A_{k}$ and
$B_{k}$, serve as state transition matrices,  which, together with the update rule in Eq. (\ref{sta_framework}), constitute the core state transformation operators.

Drawing on state-space representation as inspiration, four specialized state transformation operators are elaborately designed to facilitate the solution search for the optimization problem.\\
(1) Rotation transformation (RT)
\begin{equation}
\bm s_{k+1}= \bm s_{k} + \alpha R_r \frac{\bm u_k}{\|\bm u_k\|_2},
\end{equation}
where $\alpha > 0$ is the rotation factor;
$R_r \in \mathbb{R}$ is a uniformly distributed random variable defined over the interval [-1,1],
and $\bm u_k \in \mathbb{R}^n$ is a random vector whose entries are uniformly distributed random variables within the interval [-1,1]. This rotation transformation enables the search within a hypersphere with the maximum radius $\alpha$ (the detailed proof can be found in \cite{zhou2018dynamic}), as illustrated in Figs.~\ref{fig:state_transformations}\subref{fig:rotation_plot} and \ref{fig:state_transformations}\subref{fig:rotation_hist}.
\\
(2) Translation transformation (TT)\\
\begin{equation}
\bm s_{k+1} = \bm s_{k}+  \beta  R_{t}  \frac{\bm s_{k}- \bm s_{k-1}}{\|\bm s_{k}- \bm s_{k-1}\|_{2}},
\end{equation}
where $\beta > 0 $ is the translation factor, and $R_t \in \mathbb{R}$ is a uniformly distributed random variable over the interval $[0,1]$.
The translation transformation enables the search along the direction of $\bm s_k - \bm s_{k-1}$,  with a maximum step length of $\beta$, as illustrated in
Figs.~\ref{fig:state_transformations}\subref{fig:translation_plot} and \ref{fig:state_transformations}\subref{fig:translation_hist}.
\\
(3) Expansion transformation (ET)\\
\begin{equation}
\bm s_{k+1} = \bm s_{k}+  \gamma  R_{e} \bm s_{k},
\end{equation}
where $\gamma > 0$ is the expansion factor, and $R_{e} \in \mathbb{R}^{n \times n}$ is a random diagonal matrix whose entries follow the normal (Gaussian) distribution, which expands the entries of $\bm{s}_{k}$ to $(-\infty, +\infty)$ to enable a search over the whole space,
as illustrated in Figs.~\ref{fig:state_transformations}\subref{fig:expansion_plot} and \ref{fig:state_transformations}\subref{fig:expansion_hist}.
\\
(4) Axesion transformation (AT)\\
\begin{equation}
\bm s_{k+1} = \bm s_{k}+  \delta  R_{a}  \bm s_{k},\\
\end{equation}
where $\delta > 0$ is the axesion factor, and $R_{a} \in \mathbb{R}^{n \times n}$ is a diagonal matrix in which only one diagonal entry is nonzero and is drawn from a normal (Gaussian) distribution. This axesion transformation is designed to conduct a one-dimensional search along the coordinate axes, thereby enhancing the search capability in a single dimension, as illustrated in
Figs.~\ref{fig:state_transformations}\subref{fig:axesion_plot} and \ref{fig:state_transformations}\subref{fig:axesion_hist}.

\begin{figure*}[htbp]
	\centering
	\subfloat[(a)]{\includegraphics[width=0.35\textwidth]{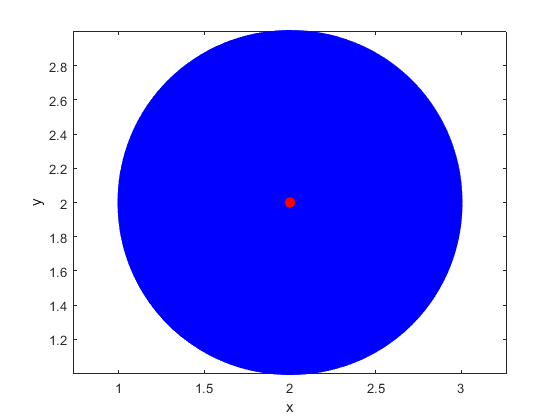}\label{fig:rotation_plot}}\hspace{0.1\textwidth}
	\subfloat[(b)]{\includegraphics[width=0.35\textwidth]{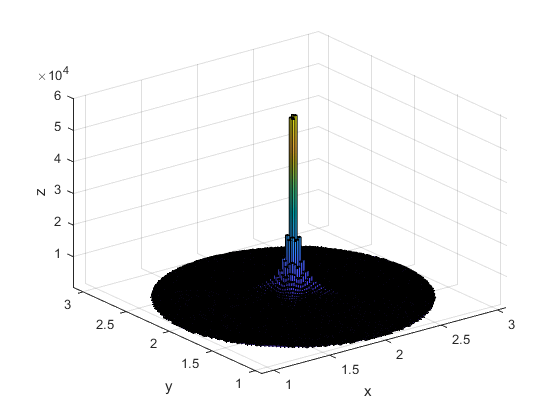}\label{fig:rotation_hist}} \\
	\subfloat[(c)]{\includegraphics[width=0.35\textwidth]{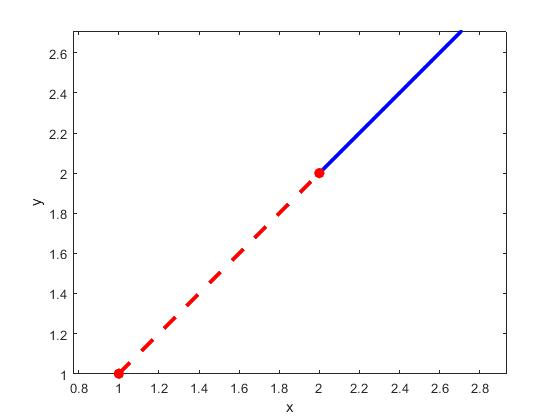}\label{fig:translation_plot}}\hspace{0.1\textwidth}
	\subfloat[(d)]{\includegraphics[width=0.35\textwidth]{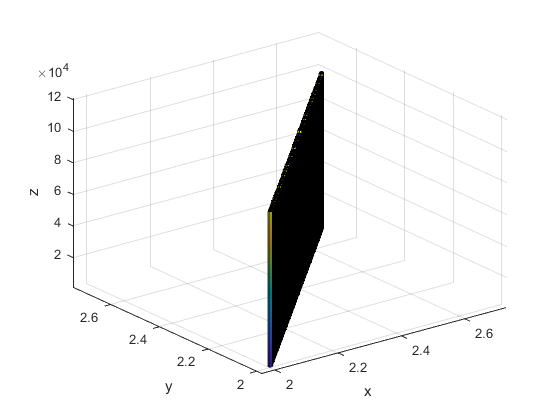}\label{fig:translation_hist}}\\
	\subfloat[(e)]{\includegraphics[width=0.35\textwidth]{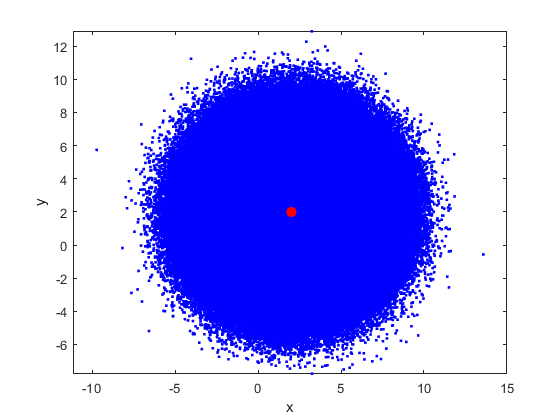}\label{fig:expansion_plot}}\hspace{0.1\textwidth}
	\subfloat[(f)]{\includegraphics[width=0.35\textwidth]{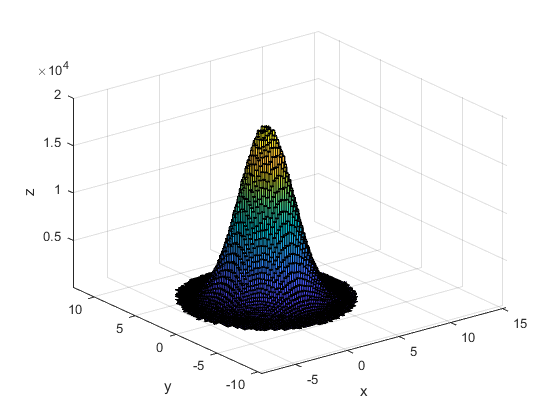}\label{fig:expansion_hist}}\\
	\subfloat[(g)]{\includegraphics[width=0.35\textwidth]{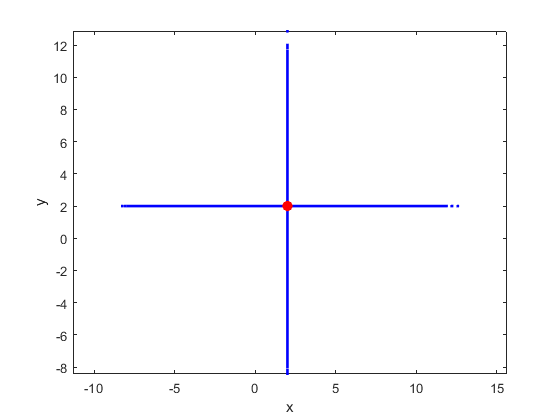}\label{fig:axesion_plot}}\hspace{0.1\textwidth}
	\subfloat[(h)]{\includegraphics[width=0.35\textwidth]{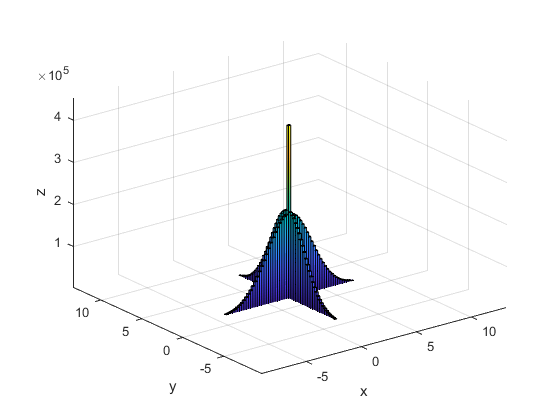}\label{fig:axesion_hist}}\\
	\caption{Visualization of 2D plots and 3D histograms of state transformations. (a) The 2D plot of rotation. (b) The 3D histogram of rotation.
            (c) The 2D plot of translation. (d) The 3D histogram of translation. (e) The 2D plot of expansion. (f) The 3D histogram of expansion. (g) The 2D plot of axesion. (h) The 3D histogram of axesion.}
	\label{fig:state_transformations}
\end{figure*}

By leveraging state transformation operators, a sampling technique, and an update strategy, the core procedure of the basic STA can be described in the pseudocode below \cite{zhou2019statistical}.
\begin{algorithmic}[1]
\Repeat
    \If{$\alpha < \alpha_{\min}$}
    \State {$\alpha \gets \alpha_{\max}$}
    \EndIf
    \State {Best $\gets$ expansion(*,Best,SE,$\beta$,$\gamma$)}
    \State {Best $\gets$ rotation(*,Best,SE,$\alpha$,$\beta$)}
    \State {Best $\gets$ axesion(*,Best,SE,$\beta$,$\delta$)}
    \State {$\alpha \gets \frac{\alpha}{\textit{fc}}$}
\Until{the termination criterion is satisfied}
\end{algorithmic}
where, \emph{SE} is the sample size, and the termination criterion is usually the maximum number of function evaluations or iterations.
It should be emphasized the TT operator is  only executed  when a new best solution is found in other state transformation, and the update strategy $f(\bm {Best}_{k+1}) \leq f(\bm {Best}_{k})$ is used to update incumbent best solution. For more details, please refer to \cite{zhou2019statistical}.

\subsection{The principle behind the STA}
The primary goal of STA is to find a global or near-global optimal solution to the optimization problem as quickly as possible.
As a constructivist learning-based intelligent optimization method, STA possesses the following five core structural components:

\begin{itemize}
    \item Globality: the capability to explore the entire search space;
    \item Optimality: the guarantee of achieving an optimal solution;
    \item Rapidity: reducing computational complexity as much as possible;
    \item Convergence: guaranteeing the convergence of the generated sequence;
    \item Controllability: flexibly managing the search space.
\end{itemize}

The globality mainly lies in the expansion transformation operator. Given that $R_{e}$ is a random diagonal
matrix with entries following a Gaussian distribution, by using the expansion transformation with a nonzero vector $\bm s_k$ and sufficiently large expansion factor $\gamma$,
the entries in $\bm s_{k+1}$ can be expanded to the range of $(-\infty, +\infty)$, from the perspective of probability theory.

The optimality mainly lies in the rotation transformation operator. It can be proved that $\|\bm s_{k+1} - \bm s_{k}\| \leq \alpha$ by using the rotation transformation \cite{zhou2012state, zhou2018dynamic}. That is to say, with a sufficiently large sample size \emph{SE} and sufficiently small rotation factor $\alpha_{\min}$, if no better candidate $\bm s_{k+1}$ is found, the incumbent $\bm{s}_{k}$ can be regarded as an optimal solution in the mathematical sense, achieved with a prescribed level of precision $\alpha_{\min}$.

The rapidity lies in several aspects: \emph{i}) the sampling technique; based on the incumbent solution, the candidates generated by each state transformation operator form a regular neighborhood, from which representative candidates can be selected by sampling, avoiding exhaustive enumeration. \emph{ii}) the alternate strategy; the alternative use of different transformation operators facilitates escape from regions around local optima.  \emph{iii}) efficient implementation; for example,  the time complexity of the rotation transformation implementation has been reduced from $O(n^2)$ to $O(n)$ \cite{zhou2018dynamic}.

The convergence mainly lies in the update rule: $f(\bm {Best}_{k+1}) \leq f(\bm {Best}_{k})$. Since the sequence $\{f({\bm{Best}}_{k})\}$ is monotonically decreasing and $f(\bm {Best})$ is bounded below, the sequence $\{f({\bm{Best}}_{k})\}$ converges.

The controllability mainly lies the transformation factors $\alpha, \beta, \gamma$ and $\delta$. These transformation factors can regulate the search space to a desired region with a specific size. For example, the rotation transformation enables the search within a hypersphere of maximal radius $\alpha$.

\begin{theorem}
If the sample size SE is sufficiently large and the rotation factor $\alpha$ is sufficiently small at the later stages, the sequence $\{f({\bm{Best}}_{k})\}$ generated by the STA can
converge to a minimum, \textit{i.e.},
\begin{eqnarray}
\lim_{k \rightarrow \infty} f(\bm {Best}_k) = f(\bm x^{*}),
\end{eqnarray}
where $\bm x^{*}$ is a minimum.
\end{theorem}

\begin{proof}
First, under the update rule $f(\bm{Best}_{k+1}) \leq f(\bm{Best}_k)$ in STA, the sequence ${f(\bm{Best}_k)}$ is monotonically nonincreasing and bounded below. Therefore, by the monotone convergence theorem, the sequence ${f(\bm{Best}_k)}$ converges.

Secondly, since the sample size \emph{SE} is sufficiently large and the rotation factor $\alpha$ is sufficiently small, it is not difficult to derive that
$f(\bm {Best}_{k}) \leq f(\bm s_{k+1}), \forall \bm s_{k+1} \in \mathcal{S}$, where $\mathcal{S} =  \{ \bm s| \|\bm s - \bm {Best}_{k}\| \leq \alpha_{\min} \}$ is the neighborhood of $\bm {Best}_{k}$, that is to say, $f(\bm x^{*}) = \lim_{k \rightarrow \infty} f(\bm {Best}_k)$ can be considered as  a minimum with solution accuracy $\alpha_{\min}$.

To sum up, the sequence $\{f({\bm{Best}}_{k})\}$ generated by the STA can converge to a minimum with certain solution accuracy.
\end{proof}

\begin{remark}
The ``guaranteed optimality" in this study means the proposed algorithms converge to a minimum, either a local minimum or a global minimum for the unconstrained optimization problem. A key drawback of many intelligent optimization algorithms is their tendency to converge prematurely to a stable point, which may not represent the true optimum \cite{zhou2025stagnation}.
\end{remark}

\begin{remark}
The sample size depends on the dimension of the optimization problem.
Based on our empirical experience, the recommended sample size SE is set to $2n$ for low dimensions, $5n$ for medium dimensions, and $10n$ for high dimensions.
\end{remark}

\subsection{Analysis of the termination criteria in intelligent optimization}
In intelligent optimization, most  studies focus on generating candidate solutions and updating incumbents to balance global exploration and local exploitation. Termination criteria are often overlooked, with a predefined number of function evaluations or iterations typically used as the stopping condition. However, termination criteria play a crucial role in solving real-world optimization problems.
If they are set too small, the algorithm may fail to find the optimal solution; on the contrary, if set too large, the algorithm may result in a waste of computing resources. As a result, extensive trial and error is required, which is not feasible in a time-constrained environment. In the academic community, the indicators to termination criteria can be summarized into the following
three categories \cite{liu2018termination, xu2025handling, hu2026auxiliary, yu2026dual}:

\begin{enumerate}
  \item Resource indicators
    \begin{itemize}
    \item Max iterations;
    \item Max function evaluations;
    \item Max CPU time.
    \end{itemize}
  \item Progress indicators
    \begin{itemize}
        \item accuracy, \emph{i.e.}, the distance to the optimum is less than a given tolerance;
        \item hitting a bound, \emph{i.e.}, the best found objective function value hits a given bound;
        \item $K$-iterations, \emph{i.e.}, there is no obvious improvement after $K$ consecutive generations or iterations.
    \end{itemize}
  \item Characteristic indicators
      \begin{itemize}
        \item convergence;
        \item diversity;
        \item clustering;
        \item statistic.
    \end{itemize}

\end{enumerate}

Although various termination criteria have been proposed in intelligent optimization, determining an appropriate termination criterion for most algorithms remains an open problem. This is because the best solution found may not be strictly optimal in a mathematical sense when computational resources are exhausted, progress has stagnated, or certain stopping conditions are met. Fortunately, as discussed earlier, STA possesses intrinsic properties that enable the design of a termination criterion with theoretical guarantees of optimality.

\section{The Proposed efficient STA}
\label{sec:ESTA}
Considering that the standard STA exhibits slow convergence in later stages, particularly on functions with flat landscapes, novel state transformations and
parameter adaptation strategy are proposed to accelerate the convergence in this section.
\subsection{Novel translation transformation based on predictive modeling}
In STA, the translation transformation is employed as a heuristic search mechanism to generate more potential solutions, and it can search along a line from previous best solution to current best solution. This is intuitive, as the current best solution improves upon the previous one.
From the perspective of regression analysis, the translation transformation in standard STA is a first-order predictive model with normalized scale, but only the last best solution is utilized. To make sufficient utilization of historical best solutions, in this part, novel translation transformations based on first and second order predictive models are proposed.

Firstly, let us recall the widely used autoregressive integrated moving average  ARIMA($p,d,q$) model \cite{li2023self}:
\begin{equation}
\bm x_t = \bm c +   \sum_{i=1}^p \phi_i \bm x_{t-i} + \bm \epsilon_t + \sum_{i=1}^q \theta_i \bm \epsilon_{t-i},
\end{equation}
where $\bm x_t$ is the actual value, $\bm \epsilon_t$ is the Gaussian white noise, $p$ and $q$ are referred to as the autoregressive and moving average orders, respectively. $d$ is the degree of differencing.

To make sure that the series can converge, new translation transformations based on the first and second order predictive models are proposed as follows
\begin{subequations}
\label{eq:predictive_model}
\begin{align}
\bm s_{k+1} &= \bm s_{k} +  \beta  \hat{R}_{t}  (\bm s_{k}- \bm s_{k-1}) ,\\
\bm s_{k+1} &= \bm s_{k} +  \beta  \hat{R}_{t}  (\bm s_{k-1}- \bm s_{k-2}) ,
\end{align}
\end{subequations}
where $\bm s_k$ represents the current state, $\bm s_{k-1}$ and $\bm s_{k-2}$ are historical states.
$\hat{R}_{t}$ $\in \mathbb{R}$ is a uniformly distributed random variable over the interval [-1,1]. Furthermore, to reflect the influence of the degree of differencing $d$, $\bm s_{k-1}$ and $\bm s_{k-2}$ are chosen randomly from the historical best archive.

The differences between the first and the second order predictive models are illustrated in Fig. \ref{fig:predictive_model}.
With the same previous best solutions $A, B, C$, the current best solution $D$, the next candidate solutions $A',B',C'$ are different using different predictive models. In the first order predictive model, $A$, $D$, and $A'$ are in the same line, the same to $B$, $D$, and $B'$, $C$, $D$, and $C'$.
In the second order predictive model, the line $AB$ is parallel to line $DA'$, the same to line $BC$ and line $DB'$, line $AC$ and line $DC'$.
Although the newly generated solutions differ, the convergence rate is governed by the same term $\beta  \hat{R}_{t}$, as shown in Eq. (\ref{eq:predictive_model}).

\begin{figure*}[!t]
	\centering
	
	\subfloat[(a)]{\includegraphics[width=0.45\textwidth]{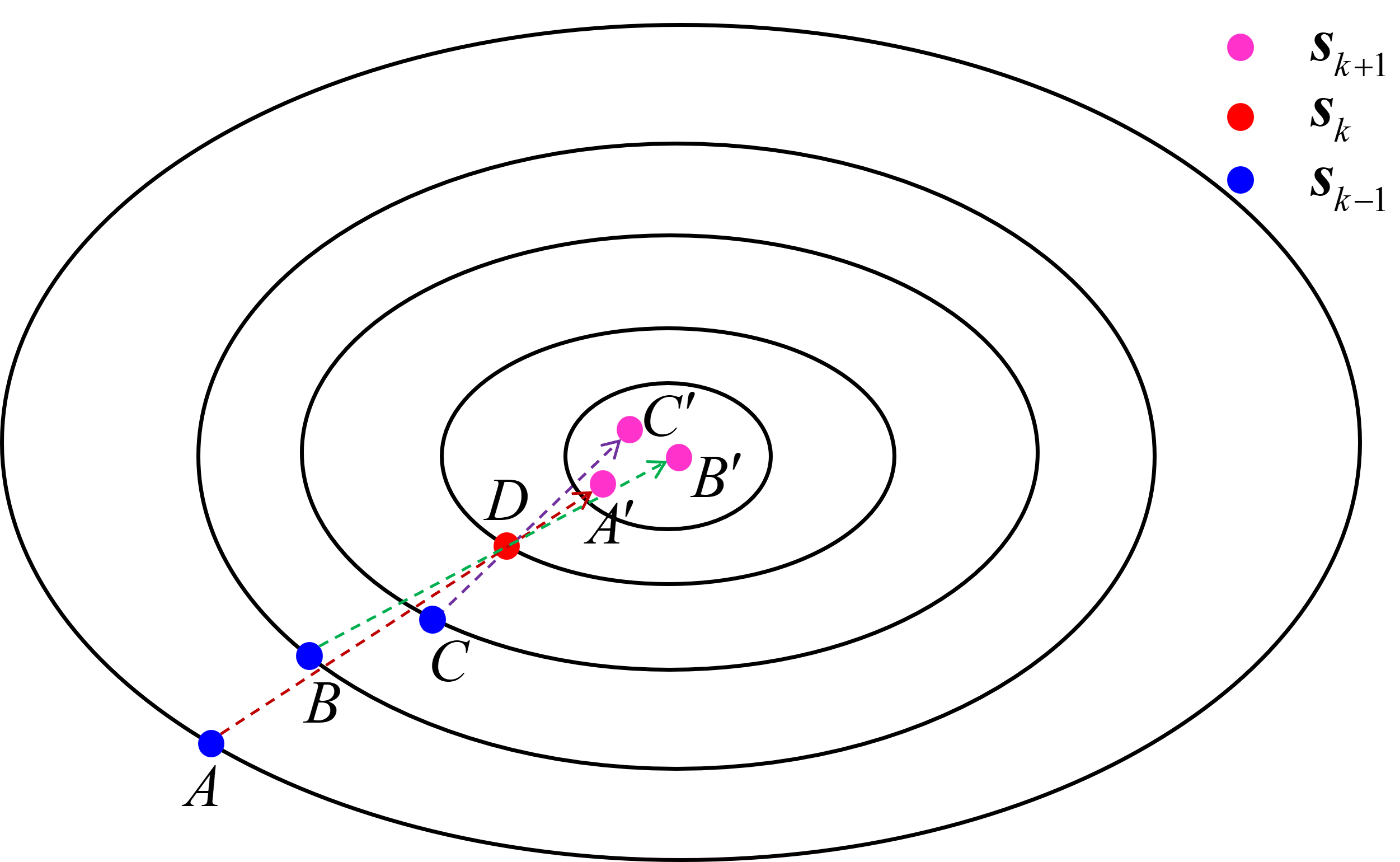}}
    \hspace{0.05\textwidth}
	\subfloat[(b)]{\includegraphics[width=0.45\textwidth]{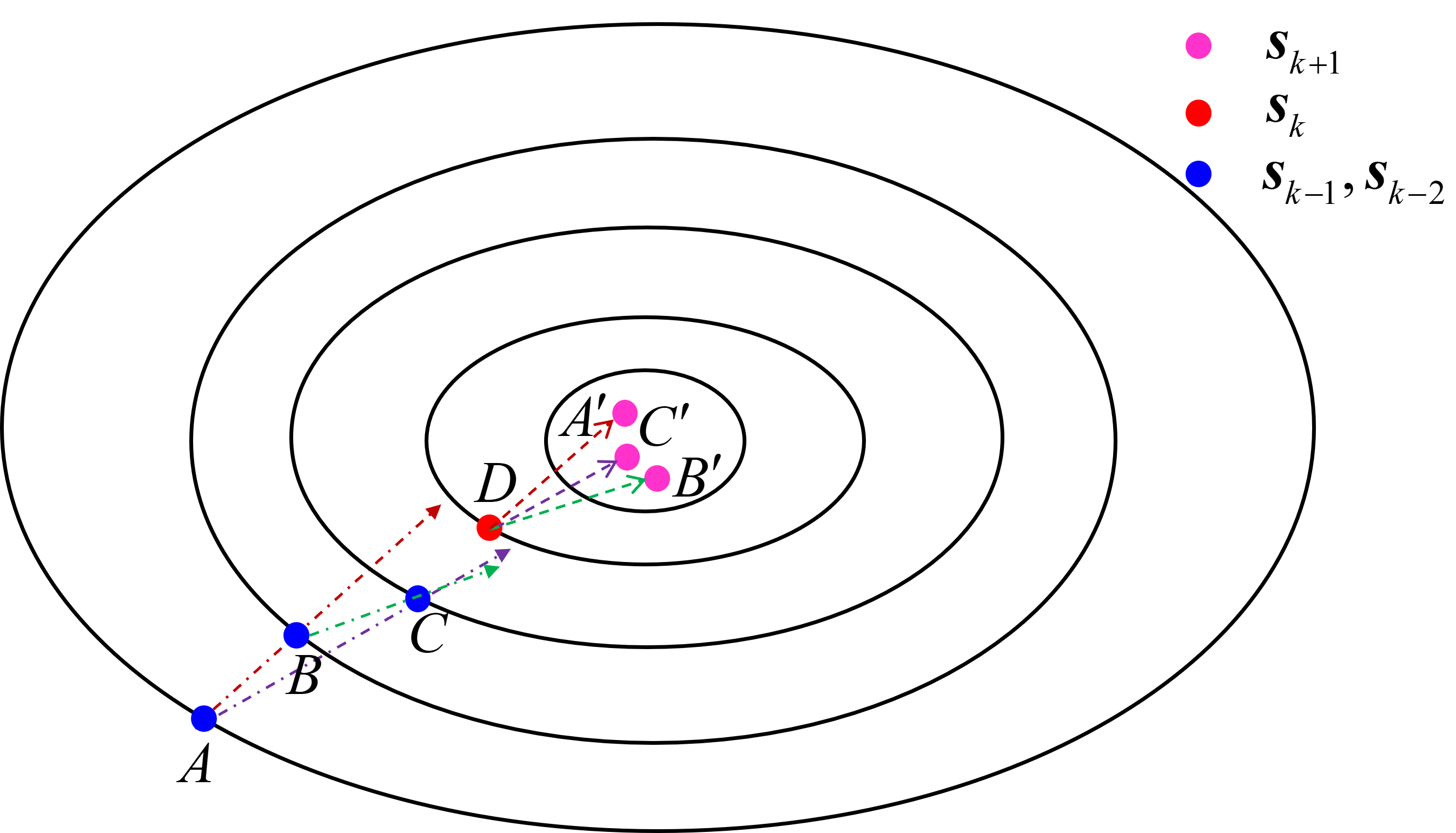}}

	\caption{Visualization and comparison of two novel translation transformations. (a) Novel translation transformation based on the first-order predictive model.
      (b) Novel translation transformation based on the second-order predictive model.}
	\label{fig:predictive_model}
\end{figure*}

\subsection{New expansion and axesion transformation for supplementation}
In the standard STA, the ET operator performs global search, whereas the AT operator conducts heuristic search. Although these two transformation operators perform well in most scenarios, their effectiveness diminishes in an extreme case. As shown in Fig. \ref{fig:extreme_case},
the global point $B$ (where $x = 2$) is far away from the local point point $A$ (where $x = 0.2$).
If the incumbent best solution $s_k$ is around the local point $A$ and the state transformation factor is set at 1, it is extremely difficult to jump out of the neighborhood since
the probability $P(s_{k+1} \geq 0.8) \approx 0.13\% $ .

\begin{figure}[!htbp]
\centering
\includegraphics[width=0.45\textwidth]{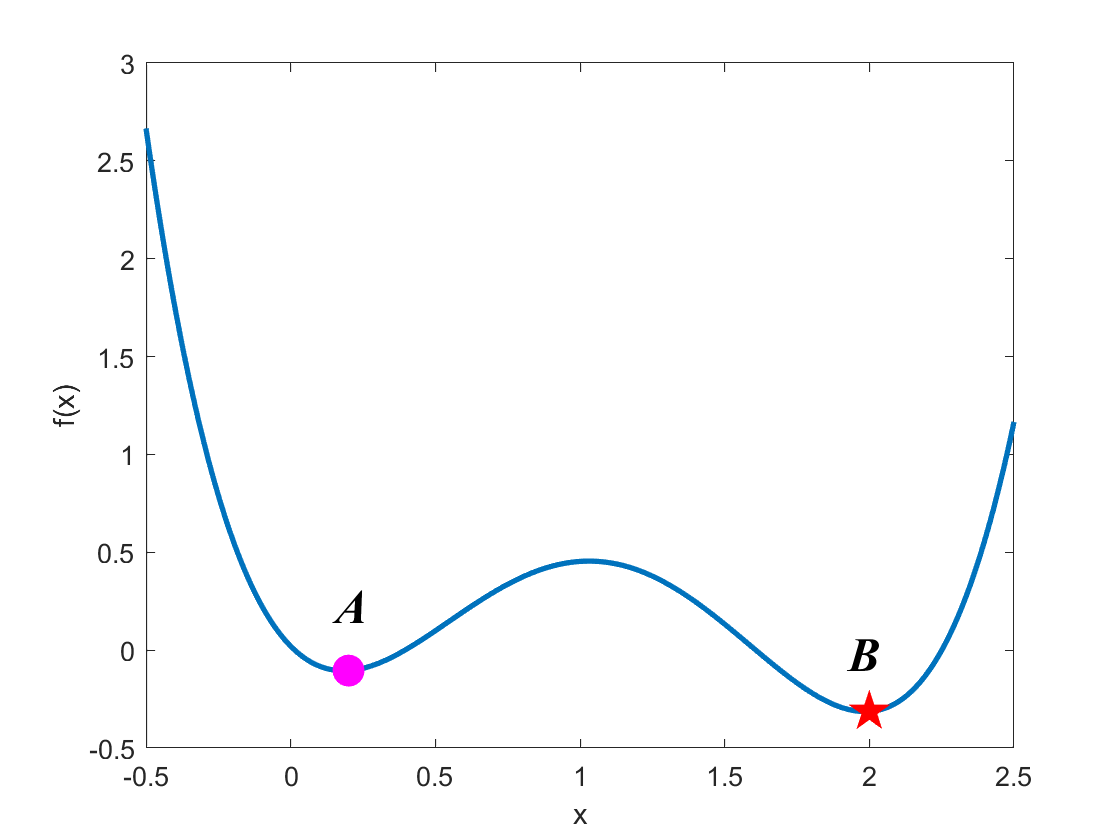}
\caption{Illustration of an extreme case.}
\label{fig:extreme_case}
\end{figure}

To resolve the above-mentioned issue, new expansion and axesion transformation are designed as follows.
\begin{equation}
\begin{aligned}
\mathrm{expansion}: \bm s_{k+1} &= \bm s_{k} + \gamma R_{e} \bm u_{k}, \\
\mathrm{axesion}: \bm s_{k+1} &= \bm s_{k} + \delta R_{a} \bm u_{k}.
\end{aligned}
\end{equation}
Here, $\bm u_{k} = [1, \cdots, 1]^T$ is the all-ones vector.

\begin{remark}
For the extreme case, by using the new axesion transformation, the probability to jump out of the local point $A$ is $P(s_{k+1} \geq 1) \approx 21\% $. Given that the original expansion and axesion transformation perform well in most cases, these new transformations are therefore combined with the original ones.
\end{remark}

\subsection{Proposed parameter control strategy}
Parameter setting is a key issue in intelligent optimization algorithms, as their performance strongly depends on parameter values. Parameter setting can be broadly categorized into parameter tuning and parameter control. In parameter tuning, suitable parameter values are selected in advance and kept fixed throughout the optimization process. In contrast, parameter control dynamically adjusts parameter values during the search to improve performance and adaptability. Since identifying optimal parameter values is often challenging and time-consuming, and no single parameter configuration is universally effective across different problems, parameter control becomes essential, particularly for complex optimization tasks.

Parameter setting is also important in STA, as the state transformation factors have been shown to significantly influence its performance. In the standard STA \cite{zhou2012state}, the translation factor $\beta$, the expansion factor $\gamma$, and the axesion factor $\delta$ are kept constant, while the rotation factor $\alpha$ decreases exponentially from 1 to $10^{-4}$ with base 2 in a periodic manner. In dynamic STA (DaSTA) \cite{zhou2018dynamic}, all state transformation factors decrease exponentially from 1 to $10^{-4}$ with base 2 in a periodic manner to enhance local exploitation in the later stages. In parameter-optimal STA (POSTA) \cite{zhou2019statistical}, a parameter selection strategy is proposed to accelerate convergence. Specifically, parameter values are selected from the set ${1, 10^{-1}, 10^{-2}, \dots, 10^{-8}}$ and kept constant for a predefined period to ensure full utilization.

\subsubsection{An intuitive parameter adaptation strategy}

First, we focus on two key state transformation factors: the rotation factor $\alpha$ and the expansion factor $\gamma$, which correspond to local search and global search, respectively. The central idea is to adapt these parameters according to changes in the incumbent solution. As is well known, larger parameter values promote global exploration, whereas smaller values favor local exploitation. Moreover, global exploration is typically emphasized in the early stages to explore unknown regions, while local exploitation is strengthened in the later stages to refine the solution. Based on our previous studies, setting the state transformation factors to 1 provides a balanced and moderate choice. If the values are too large, redundant and ineffective candidate solutions are generated, leading to inefficiency. Conversely, if the values are too small in the early stages, the improvement of the solution becomes slow.

Based on the above analysis, an intuitive parameter adaptation strategy is proposed, as illustrated in Fig. \ref{fig:param_adapt}, where the change between two consecutive incumbent solutions is defined as follows:
\begin{equation}
\triangle x = \max(\left| \bm {Best}_{k} - \bm {Best}_{k-1} \right|),
\end{equation}
here, $\bm {Best}_{k}$ and $\bm {Best}_{k-1}$ are current best solution and previous best solution, respectively. It utilizes the maximum variation in two incumbent solutions as the indicator to adjust the state transformation factors $\alpha$ and $\gamma$. The $\varepsilon$ in Fig. \ref{fig:param_adapt} represents the precision level of the solution, which is specified by the user.

\begin{figure}[!htbp]
\centering
\includegraphics[width=0.45\textwidth]{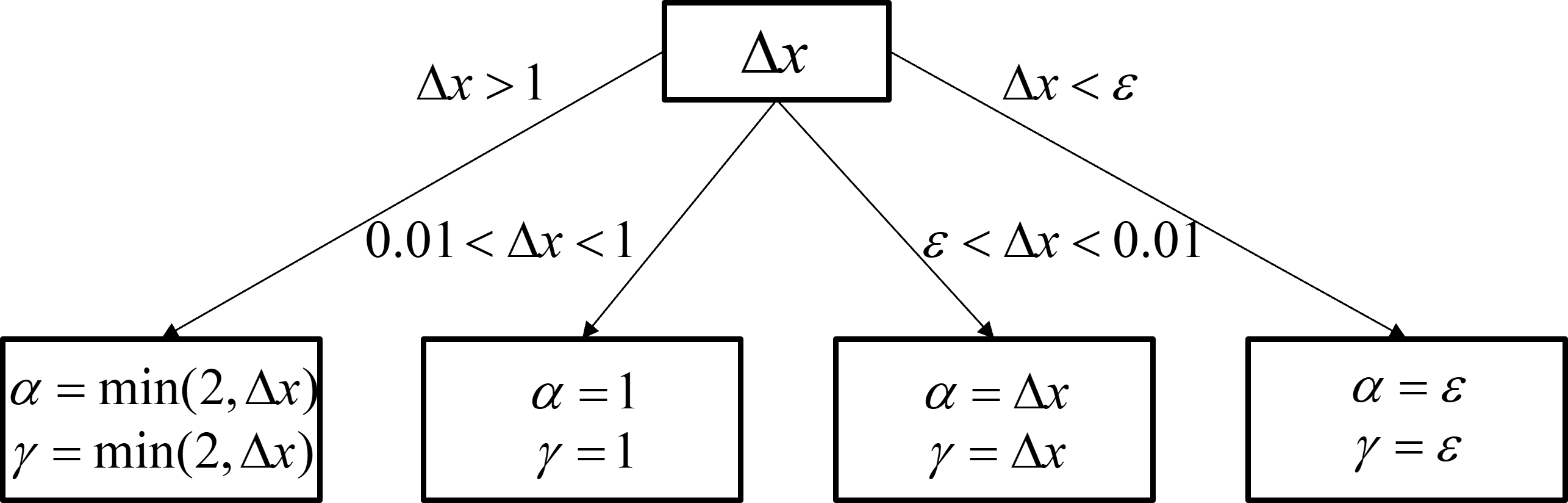}
\caption{Intuitive parameter adaptation strategy.}
\label{fig:param_adapt}
\end{figure}

\subsubsection{Parameter optimal selection strategy}
On the other hand, as indicated in POSTA \cite{zhou2019statistical}, the impact of parameter on STA is predictable, as its parameter is analogous to the learning rate, also known as the step size in gradient-based optimization algorithms, as shown below
\begin{displaymath}
\left. \begin{array}{l}
R_r \frac{\bm u_k}{\|\bm u_k\|_2} \\
\hat{R}_{t}  (\bm s_{k}- \bm s_{k-1})\\
\cdots
\end{array} \right \} \Rightarrow \bm d_k,
\left. \begin{array}{l}
\alpha \\
\beta \\
\cdots
\end{array} \right \} \Rightarrow a_k,
\bm s_{k+1} = \bm s_{k} + a_k \bm d_k,
\end{displaymath}
where $\bm d_k$ is analogous to the search direction in trust region methods.

Assuming that the incumbent solution is not optimal, sufficiently small parameter values allow the corresponding state transformation operator to generate improved candidate solutions. In contrast, overly large or inappropriate parameter values may lead to inefficient search.

With this in mind, the parameter selection problem can be viewed as analogous to line search in classical mathematical programming methods. For simplicity, an inexact line search strategy is adopted, as described below.

\begin{eqnarray}
a^{*} = \argmin_{a_k \in \Omega} f(\bm s_k + a_k\bm d_k),
\end{eqnarray}
where $\Omega = \{2, 1, 10^{-1}, 10^{-2}, \dots, 10^{-8}\}$.

\subsection{Proposed efficient STA with guaranteed optimality}
With the aforementioned novel state transformation operators and parameter control strategies,
the procedure of the proposed efficient STA with guaranteed optimality is given as follows.
\begin{algorithmic}[1]
\Repeat
    \State  prevfBest $\gets$  fBest
    \State {[Best,fBest,Archive] $\gets$ expansion(*,Best,Archive,$\gamma$)}
    \State {[Best,fBest,Archive] $\gets$ rotation(*,Best,Archive,$\alpha$)}
    \State {[Best,fBest,Archive] $\gets$ axesion(*,Best,Archive,$\delta$)}
    \State {[Best,fBest,Archive] $\gets$ translation(*,Best,Archive,$\beta$)}
    \State {[$\alpha,\beta,\gamma,\delta$] $\gets$ UpdateParameter()}
\Until{$ \mathrm{fBest} == \mathrm{prevfBest} \;\&\; \alpha \leq \varepsilon$  }
\end{algorithmic}
where, the Archive is to store historical best solutions and $\varepsilon$ is the solution accuracy.

\begin{remark}
The designed termination condition indicates that if the rotation factor $\alpha$ decreases below a predefined threshold $\varepsilon$ without any further progress ($\mathrm{fBest} == \mathrm{prevfBest}$), the STA will terminate automatically. By the way, to distinguish between different parameter control strategies, the proposed STA, which incorporates an intuitive parameter adaptation strategy that is simple yet efficient, is abbreviated as ESTA. In the meanwhile, the extended version that employs the optimal parameter selection strategy is called EXSTA.
\end{remark}

\section{Experimental results and analysis}
\subsection{Experimental setup}
In this part, extensive experiments are conducted to evaluate the performance of the proposed effective STA.
To avoid redundant comparisons among functions with similar characteristics and centre-bias issue in evolutionary computation \cite{kudela2022critical}, a set of representative benchmark functions was selected, most of which are shifted, as presented in Appendix A, Table \ref{tab:ESTA_benchmarks}.  They cover a diverse range of types, including
Type I: unimodal $(F_1, F_9, F_{11}, F_{12})$ and multimodal $(F_2 - F_8, F_{10}, F_{13}-F_{15})$; Type II: separable $(F_1, F_3, F_7, F_{13})$, and nonseparable $(F_2, F_4-F_6, F_8- F_{12}, F_{14}, F_{15})$; Type III: smooth $(F_1-F_{10})$ and unsmooth $(F_{11} - F_{15})$; Type IV: flatten $(F_2, F_7, F_8)$, and unflatten $(F_1, F_3- F_6, F_9- F_{15})$.

For a comprehensive evaluation, the proposed efficient STA is first compared with several STA variants, including the standard STA \cite{zhou2012state}, DaSTA \cite{zhou2018dynamic}, and POSTA \cite{zhou2019statistical}. It is then compared with other representative intelligent optimization algorithms, including well-established methods such as GL25 \cite{garcia2008global}, CLPSO \cite{liang2006comprehensive}, CMA-ES \cite{hansen2003reducing}, and LSHADE \cite{tanabe2013success}, as well as popular methods such as GWO \cite{mirjalili2014grey} and WOA \cite{mirjalili2016whale}, and recently proposed methods such as PGA \cite{bohat2025phototropic} and KLA \cite{ghasemi2025kirchhoff}.

For comparisons with other optimization algorithms, the experiments are divided into two groups based on termination criteria: maximum stall iterations (maxStalls) for benchmark functions and maximum function evaluations (maxFEs) for a real-world problem. Each benchmark function is tested in 20, 30, 50, and 100 dimensions, and all algorithms are independently executed 30 times in MATLAB (R2022b) on a PC equipped with an Intel(R) Core(TM) i7-1260P CPU @ 2.10 GHz.

For parameter settings, all competing optimization algorithms use their default parameter values. In STA and its variants, the sample size \emph{SE} is set to 30. For the proposed termination criterion, convergence is assumed when $\mathrm{prevBest} - \mathrm{fBest} \leq \mathrm{eps}$ (a constant in MATLAB), indicating that $\mathrm{fBest}$ remains unchanged. The solution accuracy threshold $\varepsilon$ is set to $10^{-8}$. The archive size is set to 100 in ESTA and EXSTA. When applying the optimal parameter selection strategy, the hyperparameters are set to be the same as those in POSTA.

\subsection{Comparison among different predictive models}

As described in Section \ref{sec:ESTA}, three types of predictive models are considered: the first-order model, the second-order model, and their hybrid. In this section, a comparative study is conducted to examine the differences among these models. Different ESTA variants based on these predictive models, all equipped with the proposed termination criterion, are compared, and the statistical results are reported in Table \ref{tab:ESTA_predictive_model}. The mean (denoted by \textbf{m}) and standard deviation (denoted by \textbf{sd}) of the objective function values (abbreviated as \textbf{ObjVal}) are used to evaluate the search performance and stability of the algorithms. Meanwhile, the mean gradient norm (abbreviated as \textbf{GradNorm}) is used to assess optimality.

From the statistical results in Table \ref{tab:ESTA_predictive_model}, it can be observed that all ESTA variants based on predictive models, equipped with the proposed termination criterion, achieve optimality, as the gradient norm is on the order of $10^{-4}$ or smaller (a commonly adopted tolerance in engineering applications). Most of these variants are able to locate the global optimum with satisfactory accuracy, except for the Griewank function in 20 dimensions and the Michalewicz function, for which the global optimum is unknown. Furthermore, based on the Wilcoxon rank-sum test results in Table \ref{tab:ranksum_predictive_model}, the first-order model performs better in low-dimensional problems, whereas the second-order model is more suitable for high-dimensional ones. The hybrid model, however, performs well across different dimensions and is therefore adopted in the remainder of this study.

\begin{table*}[!htbp]
	\caption{Comparison of ESTA Variants Based on Different Predictive Models with the Proposed Termination Criterion}
	\label{tab:ESTA_predictive_model}
	\tiny
	\begin{tabular}{l|c|c|c|c|c|c|c}
		\hline
		\toprule[1pt]
		\multirow{2}{*}{Fcn} & \multirow{2}{*}{Dim} & \multicolumn{2}{c|}{First-order model} & \multicolumn{2}{c|}{Second-order model} & \multicolumn{2}{c}{Hybrid model} \\
		\cline{3-8}
		& & ObjVal (m $\pm$ sd) & GradNorm (m) & ObjVal (m $\pm$ sd) & GradNorm (m) & ObjVal (m $\pm$ sd) & GradNorm (m)\\
		\hline
		& 20 & $1.74 \times 10^{-15} \pm 3.42 \times 10^{-16}$ & $8.30 \times 10^{-8}$ & $\mathbf{1.67 \times 10^{-15} \pm 2.76 \times 10^{-16}}$ & $8.14 \times 10^{-8}$ & $1.73 \times 10^{-15} \pm 4.59 \times 10^{-16}$ & $8.73 \times 10^{-8}$  \\
		$F_{1}$& 30 &$\mathbf{3.05 \times 10^{-15} \pm 4.65 \times 10^{-16}}$ & $1.12 \times 10^{-7}$ & $4.27 \times 10^{-15} \pm 1.53 \times 10^{-15}$ & $1.29 \times 10^{-7}$ & $3.29 \times 10^{-15} \pm 6.23 \times 10^{-16}$ & $1.14 \times 10^{-7}$  \\
		& 50 & $7.40 \times 10^{-15} \pm 1.59 \times 10^{-15}$ & $1.71 \times 10^{-7}$ & $1.05 \times 10^{-14} \pm 1.20 \times 10^{-15}$ & $2.04 \times 10^{-7}$ & $\mathbf{7.08 \times 10^{-15} \pm 1.06 \times 10^{-15}}$ & $1.68 \times 10^{-7}$  \\
		& 100 & $\mathbf{1.73 \times 10^{-14} \pm 1.74 \times 10^{-15}}$ & $2.63 \times 10^{-7}$ & $1.79 \times 10^{-14} \pm 3.04 \times 10^{-15}$ & $2.72 \times 10^{-7}$ & $1.97 \times 10^{-14} \pm 4.76 \times 10^{-15}$ & $2.83 \times 10^{-7}$  \\
		\hline
		& 20 & $1.63 \times 10^{-14} \pm 7.22 \times 10^{-15}$ & $4.50 \times 10^{-6}$ & $3.65 \times 10^{-14} \pm 1.96 \times 10^{-14}$ & $3.50 \times 10^{-6}$ & $\mathbf{1.33 \times 10^{-14} \pm 2.20 \times 10^{-15}}$ & $4.10 \times 10^{-6}$  \\
		$F_{2}$& 30 & $\mathbf{2.78 \times 10^{-14} \pm 1.71 \times 10^{-14}}$ & $5.26 \times 10^{-6}$ & $4.82 \times 10^{-14} \pm 2.91 \times 10^{-14}$ & $6.84 \times 10^{-6}$ & $2.21 \times 10^{-13} \pm 1.90 \times 10^{-13}$ & $4.56 \times 10^{-6}$  \\
		& 50 & $1.90 \times 10^{-13} \pm 7.75 \times 10^{-14}$ & $7.87 \times 10^{-6}$ & $2.85 \times 10^{-13} \pm 2.29 \times 10^{-13}$ & $8.83 \times 10^{-6}$ & $\mathbf{7.96 \times 10^{-14} \pm 3.84 \times 10^{-14}}$ & $7.87 \times 10^{-6}$  \\
		& 100 & $4.78 \times 10^{-11} \pm 1.64 \times 10^{-11}$ & $2.37 \times 10^{-5}$ & $\mathbf{8.97 \times 10^{-13} \pm 3.91 \times 10^{-13}}$ & $1.53 \times 10^{-5}$ & $1.79 \times 10^{-12} \pm 3.14 \times 10^{-13}$ & $1.70 \times 10^{-5}$  \\
		\hline
		& 20 & $2.84 \times 10^{-13} \pm 0 $ & $1.64 \times 10^{-5}$ & $\mathbf{1.98 \times 10^{-13} \pm 6.84 \times 10^{-14}}$ & $1.39 \times 10^{-5}$ & $2.80 \times 10^{-13} \pm 8.57 \times 10^{-14}$ & $1.64 \times 10^{-5}$  \\
		$F_{3}$& 30 & $\mathbf{1.10 \times 10^{-12} \pm 1.51 \times 10^{-13}}$ & $3.13 \times 10^{-5}$ & $2.11 \times 10^{-12} \pm 6.99 \times 10^{-13}$ & $4.12 \times 10^{-5}$ & $1.12 \times 10^{-12} \pm 1.15 \times 10^{-13}$ & $3.18 \times 10^{-5}$  \\
		& 50 & $4.59 \times 10^{-12} \pm 8.95 \times 10^{-13}$ & $6.18 \times 10^{-5}$ & $3.89 \times 10^{-12} \pm 1.51 \times 10^{-12}$ & $5.61 \times 10^{-5}$ & $\mathbf{2.87 \times 10^{-12} \pm 1.41 \times 10^{-12}}$ & $5.00 \times 10^{-5}$  \\
		& 100 & $2.87 \times 10^{-11} \pm 3.76 \times 10^{-12}$ & $1.52 \times 10^{-4}$ & $2.52 \times 10^{-11} \pm 7.03 \times 10^{-12}$ & $1.45 \times 10^{-4}$ &$\mathbf{2.37 \times 10^{-11} \pm 5.08 \times 10^{-13}}$ & $1.40 \times 10^{-4}$  \\
		\hline
		& 20 & $\mathbf{8.11 \times 10^{-3} \pm 6.86 \times 10^{-3}}$ & $2.41 \times 10^{-8}$ & $1.29 \times 10^{-2} \pm 1.44 \times 10^{-2}$ & $2.31 \times 10^{-8}$ & $7.27 \times 10^{-2} \pm 8.60 \times 10^{-2}$ & $2.88 \times 10^{-8}$  \\
		$F_{4}$& 30 & $\mathbf{4.98 \times 10^{-15} \pm 6.87 \times 10^{-16}}$ & $3.78 \times 10^{-8}$ & $1.12 \times 10^{-14} \pm 5.75 \times 10^{-15}$ & $4.17 \times 10^{-8}$ & $1.06 \times 10^{-14} \pm 5.40 \times 10^{-15}$ & $3.80 \times 10^{-8}$  \\
		& 50 & $\mathbf{1.76 \times 10^{-14} \pm 6.50 \times 10^{-15}}$ & $4.23 \times 10^{-8}$ & $2.13 \times 10^{-14} \pm 5.84 \times 10^{-15}$ & $4.38 \times 10^{-8}$ & $1.99 \times 10^{-14} \pm 1.33 \times 10^{-15}$ & $4.35 \times 10^{-8}$  \\
		& 100 & $1.13 \times 10^{-13} \pm 3.22 \times 10^{-14}$ & $7.49 \times 10^{-8}$ & $\mathbf{7.90 \times 10^{-14} \pm 1.39 \times 10^{-14}}$ & $7.61 \times 10^{-8}$ & $9.27 \times 10^{-14} \pm 3.02 \times 10^{-14}$ & $7.06 \times 10^{-8}$  \\
		\hline
		& 20 & $\mathbf{2.11 \times 10^{-9} \pm 6.36 \times 10^{-10}}$ & $2.11 \times 10^{-4}$ & $2.69 \times 10^{-9} \pm 1.69 \times 10^{-10}$ & $2.69 \times 10^{-4}$ & $2.84 \times 10^{-9} \pm 1.50 \times 10^{-10}$ & $2.84 \times 10^{-4}$  \\
		$F_{5}$& 30 & $3.32 \times 10^{-9} \pm 7.57 \times 10^{-10}$ & $3.32 \times 10^{-4}$ & $3.15 \times 10^{-9} \pm 5.98 \times 10^{-10}$ & $3.15 \times 10^{-4}$ & $\mathbf{2.07 \times 10^{-9} \pm 3.07 \times 10^{-10}}$ & $2.07 \times 10^{-4}$  \\
		& 50 & $\mathbf{3.12 \times 10^{-9} \pm 5.22 \times 10^{-10}}$ & $3.12 \times 10^{-4}$ & $5.73 \times 10^{-9} \pm 2.92 \times 10^{-9}$ & $5.73 \times 10^{-4}$ & $4.50 \times 10^{-9} \pm 1.15 \times 10^{-11}$ & $4.50 \times 10^{-4}$  \\
		& 100 & $3.84 \times 10^{-9} \pm 1.06 \times 10^{-9}$ & $2.43 \times 10^{-4}$ & $4.48 \times 10^{-9} \pm 3.35 \times 10^{-10}$ & $4.48 \times 10^{-4}$ & $\mathbf{3.72 \times 10^{-9} \pm 1.45 \times 10^{-9}}$ & $2.80 \times 10^{-4}$  \\
		\hline
		& 20 & $\mathbf{5.07 \times 10^{-15} \pm 5.13 \times 10^{-16}}$ & $5.28 \times 10^{-7}$ & $1.44 \times 10^{-14} \pm 8.06 \times 10^{-15}$ & $4.04 \times 10^{-7}$ & $8.93 \times 10^{-15} \pm 5.04 \times 10^{-15}$ & $2.98 \times 10^{-7}$  \\
		$F_{6}$& 30 & $6.40 \times 10^{-14} \pm 4.71 \times 10^{-14}$ & $7.30 \times 10^{-7}$ & $3.55 \times 10^{-14} \pm 1.44 \times 10^{-14}$ & $6.06 \times 10^{-7}$ & $\mathbf{2.35 \times 10^{-14} \pm 1.22 \times 10^{-14}}$ & $6.72 \times 10^{-7}$  \\
		& 50 & $6.66 \times 10^{-14} \pm 5.15 \times 10^{-14}$ & $1.65 \times 10^{-6}$ & $1.02 \times 10^{-13} \pm 4.75 \times 10^{-14}$ & $1.55 \times 10^{-6}$ & $\mathbf{5.68 \times 10^{-14} \pm 1.54 \times 10^{-14}}$ & $8.01 \times 10^{-7}$  \\
		& 100 & $1.33 \times 10^{-12} \pm 3.44 \times 10^{-13}$ & $2.65 \times 10^{-6}$ & $\mathbf{7.54 \times 10^{-13} \pm 1.81 \times 10^{-13}}$ & $6.87 \times 10^{-6}$ & $7.96 \times 10^{-13} \pm 3.25 \times 10^{-13}$ & $9.37 \times 10^{-6}$  \\
		\hline
		& 20 & $-1.96 \times 10^{1} \pm 5.40 \times 10^{-4}$ & $2.42 \times 10^{-5}$ & $-1.96 \times 10^{1} \pm 1.04 \times 10^{-2}$ & $1.83 \times 10^{-5}$ & $\mathbf{-1.96 \times 10^{1} \pm 7.73 \times 10^{-5}}$ & $2.15 \times 10^{-5}$  \\
		$F_{7}$& 30 & $-2.96 \times 10^{1} \pm 1.02 \times 10^{-2}$ & $6.16 \times 10^{-5}$ & $\mathbf{-2.96 \times 10^{1} \pm 3.61 \times 10^{-3}}$ & $5.41 \times 10^{-5}$ & $-2.96 \times 10^{1} \pm 5.55 \times 10^{-3}$ & $5.70 \times 10^{-5}$  \\
		& 50 & $-4.96 \times 10^{1} \pm 2.66 \times 10^{-3}$ & $2.61 \times 10^{-4}$ & $-4.96 \times 10^{1} \pm 6.72 \times 10^{-3}$ & $2.60 \times 10^{-4}$ & $\mathbf{-4.96 \times 10^{1} \pm 1.16 \times 10^{-3}}$ & $1.99 \times 10^{-4}$  \\
		& 100 & $-9.96 \times 10^{1} \pm 1.46 \times 10^{-3}$ & $1.70 \times 10^{-3}$ & $-9.96 \times 10^{1} \pm 3.02 \times 10^{-3}$ & $1.07 \times 10^{-3}$ & $\mathbf{-9.96 \times 10^{1} \pm 1.14 \times 10^{-3}}$ & $1.30 \times 10^{-3}$  \\
		\hline
		& 20 & $\mathbf{9.35 \times 10^{-10} \pm 9.83 \times 10^{-11}}$ & $4.05 \times 10^{-5}$ & $1.19 \times 10^{-9} \pm 1.04 \times 10^{-9}$ & $3.89 \times 10^{-5}$ & $1.00 \times 10^{-9} \pm 3.13 \times 10^{-10}$ & $3.94 \times 10^{-5}$  \\
		$F_{8}$& 30 & $6.65 \times 10^{-8} \pm 3.79 \times 10^{-8}$ & $1.67 \times 10^{-4}$ & $\mathbf{2.47 \times 10^{-8} \pm 8.85 \times 10^{-9}}$ & $1.38 \times 10^{-4}$ & $3.53 \times 10^{-8} \pm 8.36 \times 10^{-9}$ & $1.56 \times 10^{-4}$  \\
		& 50 & $7.01 \times 10^{-6} \pm 3.82 \times 10^{-6}$ & $1.28 \times 10^{-3}$ & $3.46 \times 10^{-6} \pm 1.45 \times 10^{-6}$ & $1.23 \times 10^{-3}$ & $\mathbf{2.27 \times 10^{-6} \pm 1.18 \times 10^{-6}}$ & $1.27 \times 10^{-3}$  \\
		& 100 & $2.74 \times 10^{-2} \pm 3.73 \times 10^{-3}$ & $2.74 \times 10^{-2}$ & $\mathbf{1.23 \times 10^{-3} \pm 8.91 \times 10^{-4}}$ & $2.93 \times 10^{-2}$ & $3.28 \times 10^{-3} \pm 1.52 \times 10^{-3}$ & $2.52 \times 10^{-2}$  \\
		\hline
		& 20 & $\mathbf{7.76 \times 10^{-15} \pm 3.46 \times 10^{-15}}$ & $2.68 \times 10^{-7}$ & $8.58 \times 10^{-15} \pm 2.10 \times 10^{-15}$ & $1.99 \times 10^{-7}$ & $7.86 \times 10^{-15} \pm 3.70 \times 10^{-15}$ & $1.71 \times 10^{-7}$  \\
		$F_{9}$& 30 & $\mathbf{1.25 \times 10^{-14} \pm 1.81 \times 10^{-15}}$ & $3.99 \times 10^{-7}$ & $1.47 \times 10^{-14} \pm 2.66 \times 10^{-15}$ & $2.51 \times 10^{-7}$ & $2.22 \times 10^{-14} \pm 1.20 \times 10^{-14}$ & $3.01 \times 10^{-7}$  \\
		& 50 & $7.55 \times 10^{-14} \pm 3.06 \times 10^{-14}$ & $7.56 \times 10^{-7}$ & $\mathbf{4.02 \times 10^{-14} \pm 9.35 \times 10^{-15}}$ & $5.90 \times 10^{-7}$ & $6.74 \times 10^{-14} \pm 1.81 \times 10^{-14}$ & $5.80 \times 10^{-7}$  \\
		& 100 & $4.07 \times 10^{-13} \pm 1.56 \times 10^{-13}$ & $1.70 \times 10^{-6}$ & $\mathbf{2.89 \times 10^{-13} \pm 1.59 \times 10^{-13}}$ & $1.46 \times 10^{-6}$ & $3.04 \times 10^{-13} \pm 4.32 \times 10^{-14}$ & $1.25 \times 10^{-6}$  \\
		\hline
		& 20 & $\mathbf{1.18 \times 10^{-15} \pm 2.26 \times 10^{-16}}$ & $1.56 \times 10^{-7}$ & $1.83 \times 10^{-15} \pm 4.11 \times 10^{-16}$ & $1.72 \times 10^{-7}$ & $2.22 \times 10^{-15} \pm 9.92 \times 10^{-16}$ & $2.02 \times 10^{-7}$  \\
		$F_{10}$& 30 & $4.13 \times 10^{-15} \pm 1.13 \times 10^{-15}$ & $2.74 \times 10^{-7}$ & $\mathbf{3.84 \times 10^{-15} \pm 1.46 \times 10^{-15}}$ & $2.66 \times 10^{-7}$ & $5.19 \times 10^{-15} \pm 2.30 \times 10^{-15}$ & $2.75 \times 10^{-7}$  \\
		& 50 & $1.22 \times 10^{-14} \pm 6.24 \times 10^{-15}$ & $4.32 \times 10^{-7}$ & $\mathbf{6.37 \times 10^{-15} \pm 3.74 \times 10^{-15}}$ & $3.22 \times 10^{-7}$ & $1.05 \times 10^{-14} \pm 8.88 \times 10^{-16}$ & $4.11 \times 10^{-7}$  \\
		& 100 & $2.57 \times 10^{-14} \pm 3.95 \times 10^{-15}$ & $5.73 \times 10^{-7}$ & $\mathbf{1.75 \times 10^{-14} \pm 7.55 \times 10^{-16}}$ & $5.07 \times 10^{-7}$ & $2.28 \times 10^{-14} \pm 2.23 \times 10^{-15}$ & $5.47 \times 10^{-7}$  \\
		\bottomrule[1pt]
		\hline
	\end{tabular}
\end{table*}

\begin{table*}[!htbp]
	\centering
	\caption{Wilcoxon Rank-Sum Test Results of Different Predictive Models}
	\label{tab:ranksum_predictive_model}
	\begin{threeparttable}
		\renewcommand{\arraystretch}{1.2}
		\begin{tabular}{lccc ccc ccc ccc}
			\toprule
			Type & \multicolumn{3}{c}{20D} & \multicolumn{3}{c}{30D} & \multicolumn{3}{c}{50D} & \multicolumn{3}{c}{100D} \\
			\cmidrule(lr){2-4} \cmidrule(lr){5-7} \cmidrule(lr){8-10} \cmidrule(lr){11-13}
			& \#\texttt{+} & \#\texttt{=} & \#\texttt{-}
			& \#\texttt{+} & \#\texttt{=} & \#\texttt{-}
			& \#\texttt{+} & \#\texttt{=} & \#\texttt{-}
			& \#\texttt{+} & \#\texttt{=} & \#\texttt{-} \\
			\midrule
			First-order model      &  5 & 15 & 0 & 2 & 15 & 3 & 2 & 11 & 7 & 0 & 12 & 8\\
			Second-order model      &  2 & 16 & 2 & 3 & 15 & 2 & 4 & 12 & 4 & 6 & 14 & 0\\
			Hybrid model         &  0 & 15 & 5 & 2 & 16 & 2 & 7 & 11 & 2 & 4 & 14 & 2\\
			\bottomrule
		\end{tabular}
		\begin{tablenotes}
			\footnotesize
			\item 	Note: \texttt{\#+}, \texttt{\#=}, and \texttt{\#-} represent the count of instances where one model performs significantly better than, are statistically equivalent to, and are significantly worse than its counterparts, respectively.
		\end{tablenotes}
	\end{threeparttable}
\end{table*}

\subsection{Comparison with STA  variants}
In this subsection, the proposed effective STAs are compared with STA variants. The termination condition is set to maxFEs = 2e3$n \log_2(n)$.
The experimental results are shown in Table \ref{tab:STA_variants}, Table \ref{tab:ranksum_STA_variants} and Fig. \ref{fig:STA_variants}. It is observed that, with the prescribed maxFEs, the standard STA performs well on $F_1, F_3- F_6$ and $F_9- F_{15}$, but not on $F_2, F_7$ and$F_8$. In other words, the standard STA converges slowly on flat functions.
The DaSTA enhances the standard STA in both search capability and solution accuracy on non-flat functions, but it requires more time due to the use of probability in accepting a worse solution. Compared with standard STA, the POSTA alleviates the difficulty with flat functions and also improves solution accuracy on non-flat functions.
Compared with standard STA and its other variants, it can be observed that the proposed efficient STAs exhibit a markedly improved convergence rate in the later stages, especially for the flat functions ($F_2, F_7$ and $F_8$), primarily attributed to the proposed new translation transformation based on predictive modeling.

The proposed efficient STA variants achieve superior overall performance, including enhanced search capability, convergence rate, and solution accuracy.
Notably, ESTA and EXSTA outperform other STA variants in all cases.

\begin{table*}[!htbp]
	\centering
	\caption{Performance Comparison of STA Variants under the maxFEs Termination Criterion}
	\label{tab:STA_variants}
	\resizebox{\textwidth}{!}{%
		{\fontsize{4pt}{6pt}\selectfont
			\begin{tabular}{l|c|c|c|c|c|c}
				\hline
				\toprule[1pt]
				Fcn & Dim & STA & DaSTA & POSTA & ESTA & EXSTA\\
				\hline
				& 20 & $8.04 \times 10^{-12} \pm 2.75 \times 10^{-12}$ & $7.63 \times 10^{-18} \pm 2.86 \times 10^{-18}$ & $3.86 \times 10^{-20} \pm 1.51 \times 10^{-20}$ & $1.34 \times 10^{-18} \pm 7.43 \times 10^{-19}$ & $\mathbf{1.06 \times 10^{-22} \pm 8.15 \times 10^{-23}}$ \\
				$F_{1}$& 30 & $1.12 \times 10^{-11} \pm 2.79 \times 10^{-12}$ & $9.51 \times 10^{-18} \pm 2.81 \times 10^{-18}$ & $4.60 \times 10^{-20} \pm 1.35 \times 10^{-20}$ & $2.09 \times 10^{-18} \pm 8.84 \times 10^{-19}$ & $\mathbf{7.62 \times 10^{-22} \pm 6.23 \times 10^{-22}}$ \\
				& 50 & $1.36 \times 10^{-11} \pm 2.70 \times 10^{-12}$ & $1.31 \times 10^{-17} \pm 3.84 \times 10^{-18}$ & $7.47 \times 10^{-20} \pm 1.79 \times 10^{-20}$ & $4.03 \times 10^{-18} \pm 1.51 \times 10^{-18}$ & $\mathbf{1.16 \times 10^{-20} \pm 5.18 \times 10^{-21}}$ \\
				& 100 & $1.97 \times 10^{-11} \pm 3.02 \times 10^{-12}$ & $1.80 \times 10^{-17} \pm 3.41 \times 10^{-18}$ & $1.13 \times 10^{-19} \pm 2.42 \times 10^{-20}$ & $7.21 \times 10^{-18} \pm 1.94 \times 10^{-18}$ & $\mathbf{3.48 \times 10^{-20} \pm 9.37 \times 10^{-21}}$ \\
				\hline
				& 20 & $8.57 \pm 5.53$ & $1.28 \times 10^{1} \pm 2.65$ & $1.15 \times 10^{1} \pm 2.94$ & $2.17 \times 10^{-14} \pm 2.13 \times 10^{-14}$ & $\mathbf{5.14 \times 10^{-16} \pm 8.22 \times 10^{-16}}$ \\
				$F_{2}$& 30 & $5.08 \times 10^{1} \pm 3.46 \times 10^{1}$ & $1.96 \times 10^{1} \pm 3.71$ & $2.05 \times 10^{1} \pm 2.06$ & $9.18 \times 10^{-14} \pm 1.60 \times 10^{-13}$ & $\mathbf{5.36 \times 10^{-17} \pm 9.21 \times 10^{-17}}$ \\
				& 50 & $8.87 \times 10^{1} \pm 5.10 \times 10^{1}$ & $3.46 \times 10^{1} \pm 7.64$ & $3.00 \times 10^{1} \pm 1.42 \times 10^{1}$ & $1.45 \times 10^{-13} \pm 1.27 \times 10^{-13}$ & $\mathbf{3.81 \times 10^{-16} \pm 4.97 \times 10^{-16}}$ \\
				& 100 & $2.29 \times 10^{2} \pm 9.65 \times 10^{1}$ & $8.59 \times 10^{1} \pm 2.69 \times 10^{1}$ & $8.26 \times 10^{1} \pm 3.94 \times 10^{1}$ & $1.21 \times 10^{-11} \pm 2.12 \times 10^{-11}$ & $\mathbf{3.00 \times 10^{-15} \pm 6.83 \times 10^{-15}}$ \\
				\hline
				& 20 & $3.14 \times 10^{-9} \pm 1.49 \times 10^{-9}$ & $3.65 \times 10^{-14} \pm 2.66 \times 10^{-14}$ & $3.04 \times 10^{-14} \pm 3.31 \times 10^{-14}$ & $3.76 \times 10^{-14} \pm 2.79 \times 10^{-14}$ & $\mathbf{0 \pm 0}$ \\
				$F_{3}$& 30 & $8.65 \times 10^{-9} \pm 4.61 \times 10^{-9}$ & $1.19 \times 10^{-13} \pm 5.85 \times 10^{-14}$ & $\mathbf{1.14 \times 10^{-13} \pm 0}$ & $1.59 \times 10^{-13} \pm 6.23 \times 10^{-14}$ & $1.35 \times 10^{-13} \pm 7.23 \times 10^{-14}$ \\
				& 50 & $3.95 \times 10^{-8} \pm 5.24 \times 10^{-8}$ & $2.82 \times 10^{-13} \pm 1.30 \times 10^{-13}$ & $2.10 \times 10^{-13} \pm 7.78 \times 10^{-14}$ & $3.39 \times 10^{-13} \pm 1.75 \times 10^{-13}$ & $\mathbf{1.75 \times 10^{-13} \pm 6.61 \times 10^{-14}}$ \\
				& 100 & $2.31 \pm 8.92 \times 10^{-1}$ & $2.29 \times 10^{-12} \pm 4.89 \times 10^{-13}$ & $\mathbf{8.59 \times 10^{-13} \pm 2.32 \times 10^{-13}}$ & $2.99 \times 10^{-12} \pm 8.24 \times 10^{-13}$ & $1.03 \times 10^{-12} \pm 2.93 \times 10^{-13}$ \\
				\hline
				& 20 & $3.97 \times 10^{-2} \pm 6.42 \times 10^{-2}$ & $6.52 \times 10^{-2} \pm 6.42 \times 10^{-2}$ & $2.29 \times 10^{-1} \pm 2.77 \times 10^{-1}$ & $9.16 \times 10^{-3} \pm 1.50 \times 10^{-2}$ & $\mathbf{0 \pm 0}$ \\
				$F_{4}$& 30 & $2.83 \times 10^{-12} \pm 9.22 \times 10^{-13}$ & $1.89 \times 10^{-2} \pm 1.88 \times 10^{-2}$ & $\mathbf{0 \pm 0}$ & $5.38 \times 10^{-16} \pm 2.38 \times 10^{-16}$ & $\mathbf{0 \pm 0}$ \\
				& 50 & $3.35 \times 10^{-12} \pm 8.85 \times 10^{-13}$ & $1.51 \times 10^{-2} \pm 1.74 \times 10^{-2}$ & $\mathbf{0 \pm 0}$ & $1.70 \times 10^{-15} \pm 6.61 \times 10^{-16}$ & $\mathbf{0 \pm 0}$ \\
				& 100 & $3.95 \times 10^{-12} \pm 7.01 \times 10^{-13}$ & $1.03 \times 10^{-2} \pm 9.24 \times 10^{-3}$ & $\mathbf{0 \pm 0}$ & $8.68 \times 10^{-15} \pm 2.68 \times 10^{-15}$ & $\mathbf{0 \pm 0}$ \\
				\hline
				& 20 & $2.88 \times 10^{-6} \pm 3.25 \times 10^{-7}$ & $3.32 \times 10^{-9} \pm 8.38 \times 10^{-10}$ & $2.01 \times 10^{-10} \pm 4.20 \times 10^{-11}$ & $1.12 \times 10^{-9} \pm 3.93 \times 10^{-10}$ & $\mathbf{1.55 \times 10^{-11} \pm 7.27 \times 10^{-12}}$ \\
				$F_{5}$& 30 & $2.73 \times 10^{-6} \pm 4.91 \times 10^{-7}$ & $3.49 \times 10^{-9} \pm 1.01 \times 10^{-9}$ & $1.92 \times 10^{-10} \pm 2.93 \times 10^{-11}$ & $1.18 \times 10^{-9} \pm 2.98 \times 10^{-10}$ & $\mathbf{2.00 \times 10^{-11} \pm 7.14 \times 10^{-12}}$ \\
				& 50 & $2.48 \times 10^{-6} \pm 2.73 \times 10^{-7}$ & $5.68 \times 10^{-9} \pm 1.90 \times 10^{-9}$ & $2.73 \times 10^{-10} \pm 3.63 \times 10^{-11}$ & $1.21 \times 10^{-9} \pm 1.64 \times 10^{-10}$ & $\mathbf{6.50 \times 10^{-11} \pm 1.08 \times 10^{-11}}$ \\
				& 100 & $3.67 \times 10^{-6} \pm 5.89 \times 10^{-7}$ & $2.00 \times 10^{1} \pm 2.78 \times 10^{-3}$ & $2.00 \times 10^{1} \pm 5.31 \times 10^{-5}$ & $3.82 \times 10^{-9} \pm 1.83 \times 10^{-9}$ & $\mathbf{1.61 \times 10^{-10} \pm 5.08 \times 10^{-11}}$ \\
				\hline
				& 20 & $1.70 \times 10^{-9} \pm 7.33 \times 10^{-10}$ & $1.31 \times 10^{-10} \pm 1.23 \times 10^{-10}$ & $3.68 \times 10^{-18} \pm 1.37 \times 10^{-18}$ & $8.26 \times 10^{-17} \pm 5.54 \times 10^{-17}$ & $\mathbf{1.04 \times 10^{-19} \pm 7.43 \times 10^{-20}}$ \\
				$F_{6}$& 30 & $2.20 \times 10^{-8} \pm 8.97 \times 10^{-9}$ & $4.85 \times 10^{-7} \pm 3.77 \times 10^{-7}$ & $3.68 \times 10^{-17} \pm 1.34 \times 10^{-17}$ & $3.15 \times 10^{-16} \pm 1.29 \times 10^{-16}$ & $\mathbf{1.58 \times 10^{-18} \pm 8.80 \times 10^{-19}}$ \\
				& 50 & $1.31 \times 10^{-3} \pm 1.03 \times 10^{-3}$ & $7.40 \times 10^{-3} \pm 6.48 \times 10^{-3}$ & $2.91 \times 10^{-11} \pm 2.57 \times 10^{-11}$ & $2.58 \times 10^{-15} \pm 8.29 \times 10^{-16}$ & $\mathbf{3.88 \times 10^{-17} \pm 1.59 \times 10^{-17}}$ \\
				& 100 & $6.17 \times 10^{4} \pm 1.33 \times 10^{4}$ & $1.33 \times 10^{4} \pm 5.14 \times 10^{3}$ & $3.33 \times 10^{1} \pm 3.01 \times 10^{1}$ & $\mathbf{6.92 \times 10^{-14} \pm 2.00 \times 10^{-14}}$ & $1.43 \times 10^{-13} \pm 8.55 \times 10^{-14}$ \\
				\hline
				& 20 & $-1.86 \times 10^{1} \pm 1.03$ & $-1.88 \times 10^{1} \pm 9.30 \times 10^{-1}$ & $-1.86 \times 10^{1} \pm 1.32$ & $-1.96 \times 10^{1} \pm 7.44 \times 10^{-3}$ & $\mathbf{-1.96 \times 10^{1} \pm 2.63 \times 10^{-3}}$ \\
				$F_{7}$& 30 & $-2.89 \times 10^{1} \pm 6.72 \times 10^{-1}$ & $-2.79 \times 10^{1} \pm 1.26$ & $-2.89 \times 10^{1} \pm 7.36 \times 10^{-1}$ & $\mathbf{-2.96 \times 10^{1} \pm 1.47 \times 10^{-2}}$ & $-2.96 \times 10^{1} \pm 1.51 \times 10^{-2}$ \\
				& 50 & $-4.86 \times 10^{1} \pm 8.59 \times 10^{-1}$ & $-4.70 \times 10^{1} \pm 1.55$ & $-4.86 \times 10^{1} \pm 8.84 \times 10^{-1}$ & $\mathbf{-4.96 \times 10^{1} \pm 1.69 \times 10^{-2}}$ & $-4.96 \times 10^{1} \pm 1.71 \times 10^{-2}$ \\
				& 100 & $-9.82 \times 10^{1} \pm 1.26$ & $-9.62 \times 10^{1} \pm 1.75$ & $-9.83 \times 10^{1} \pm 7.07 \times 10^{-1}$ & $-9.95 \times 10^{1} \pm 1.86 \times 10^{-2}$ & $\mathbf{-9.96 \times 10^{1} \pm 1.54 \times 10^{-2}}$ \\
				\hline
				& 20 & $1.61 \pm 1.78$ & $2.94 \times 10^{-2} \pm 2.04 \times 10^{-2}$ & $1.47 \times 10^{-2} \pm 9.11 \times 10^{-3}$ & $2.96 \times 10^{-10} \pm 2.04 \times 10^{-10}$ & $\mathbf{4.99 \times 10^{-11} \pm 4.81 \times 10^{-11}}$ \\
				$F_{8}$& 30 & $2.06 \times 10^{2} \pm 1.77 \times 10^{2}$ & $1.20 \times 10^{1} \pm 1.01 \times 10^{1}$ & $1.06 \times 10^{1} \pm 8.79$ & $8.31 \times 10^{-9} \pm 4.60 \times 10^{-9}$ & $\mathbf{1.29 \times 10^{-9} \pm 9.33 \times 10^{-10}}$ \\
				& 50 & $6.28 \times 10^{3} \pm 4.11 \times 10^{3}$ & $1.17 \times 10^{3} \pm 8.71 \times 10^{2}$ & $2.68 \times 10^{2} \pm 2.82 \times 10^{2}$ & $1.28 \times 10^{-6} \pm 8.25 \times 10^{-7}$ & $\mathbf{1.13 \times 10^{-7} \pm 6.30 \times 10^{-8}}$ \\
				& 100 & $3.57 \times 10^{5} \pm 3.55 \times 10^{5}$ & $6.49 \times 10^{4} \pm 2.68 \times 10^{4}$ & $2.78 \times 10^{4} \pm 2.45 \times 10^{4}$ & $2.67 \times 10^{-2} \pm 2.42 \times 10^{-2}$ & $\mathbf{7.56 \times 10^{-5} \pm 4.93 \times 10^{-5}}$ \\
				\hline
				& 20 & $1.08 \times 10^{-7} \pm 6.31 \times 10^{-8}$ & $2.15 \times 10^{-5} \pm 1.86 \times 10^{-5}$ & $4.91 \times 10^{-11} \pm 5.47 \times 10^{-11}$ & $1.49 \times 10^{-16} \pm 9.39 \times 10^{-17}$ & $\mathbf{3.75 \times 10^{-19} \pm 1.76 \times 10^{-19}}$ \\
				$F_{9}$& 30 & $1.15 \times 10^{-5} \pm 4.89 \times 10^{-6}$ & $5.36 \times 10^{-4} \pm 2.43 \times 10^{-4}$ & $6.77 \times 10^{-8} \pm 6.07 \times 10^{-8}$ & $9.43 \times 10^{-16} \pm 4.56 \times 10^{-16}$ & $\mathbf{6.32 \times 10^{-18} \pm 3.19 \times 10^{-18}}$ \\
				& 50 & $4.52 \times 10^{-3} \pm 1.99 \times 10^{-3}$ & $2.39 \times 10^{-2} \pm 9.35 \times 10^{-3}$ & $1.14 \times 10^{-4} \pm 7.43 \times 10^{-5}$ & $8.52 \times 10^{-15} \pm 3.24 \times 10^{-15}$ & $\mathbf{8.06 \times 10^{-17} \pm 2.15 \times 10^{-17}}$ \\
				& 100 & $8.46 \pm 2.74$ & $3.56 \pm 1.03$ & $1.41 \times 10^{-1} \pm 4.27 \times 10^{-2}$ & $2.84 \times 10^{-13} \pm 5.21 \times 10^{-14}$ & $\mathbf{1.50 \times 10^{-15} \pm 3.57 \times 10^{-16}}$ \\
				\hline
				& 20 & $2.18 \times 10^{-10} \pm 1.21 \times 10^{-10}$ & $4.17 \times 10^{-14} \pm 3.60 \times 10^{-14}$ & $7.63 \times 10^{-19} \pm 3.09 \times 10^{-19}$ & $1.03 \times 10^{-17} \pm 5.62 \times 10^{-18}$ & $\mathbf{1.05 \times 10^{-21} \pm 9.21 \times 10^{-22}}$ \\
				$F_{10}$& 30 & $3.80 \times 10^{-10} \pm 1.37 \times 10^{-10}$ & $6.20 \times 10^{-14} \pm 4.61 \times 10^{-14}$ & $1.57 \times 10^{-18} \pm 5.88 \times 10^{-19}$ & $2.29 \times 10^{-17} \pm 1.36 \times 10^{-17}$ & $\mathbf{3.51 \times 10^{-21} \pm 2.25 \times 10^{-21}}$ \\
				& 50 & $5.80 \times 10^{-10} \pm 1.94 \times 10^{-10}$ & $6.41 \times 10^{-14} \pm 3.32 \times 10^{-14}$ & $2.17 \times 10^{-18} \pm 5.75 \times 10^{-19}$ & $4.10 \times 10^{-17} \pm 1.57 \times 10^{-17}$ & $\mathbf{7.43 \times 10^{-20} \pm 2.82 \times 10^{-20}}$ \\
				& 100 & $1.13 \times 10^{-9} \pm 2.02 \times 10^{-10}$ & $1.15 \times 10^{-13} \pm 4.19 \times 10^{-14}$ & $4.34 \times 10^{-18} \pm 8.92 \times 10^{-19}$ & $7.28 \times 10^{-17} \pm 2.06 \times 10^{-17}$ & $\mathbf{3.34 \times 10^{-19} \pm 1.34 \times 10^{-19}}$ \\
				\hline
				& 20 & $2.03 \times 10^{-2} \pm 9.95 \times 10^{-3}$ & $1.85 \times 10^{-5} \pm 1.30 \times 10^{-5}$ & $3.29 \times 10^{-8} \pm 2.52 \times 10^{-8}$ & $2.98 \times 10^{-8} \pm 1.44 \times 10^{-8}$ & $\mathbf{4.76 \times 10^{-9} \pm 3.51 \times 10^{-9}}$ \\
				$F_{11}$& 30 & $1.04 \times 10^{-1} \pm 3.33 \times 10^{-2}$ & $6.60 \times 10^{-4} \pm 3.86 \times 10^{-4}$ & $1.17 \times 10^{-4} \pm 1.14 \times 10^{-4}$ & $\mathbf{1.92 \times 10^{-7} \pm 8.18 \times 10^{-8}}$ & $1.41 \times 10^{-6} \pm 1.04 \times 10^{-6}$ \\
				& 50 & $3.54 \times 10^{-1} \pm 7.46 \times 10^{-2}$ & $1.60 \times 10^{-1} \pm 1.32 \times 10^{-1}$ & $2.81 \times 10^{-2} \pm 1.48 \times 10^{-2}$ & $\mathbf{2.13 \times 10^{-3} \pm 2.60 \times 10^{-3}}$ & $2.43 \times 10^{-3} \pm 1.41 \times 10^{-3}$ \\
				& 100 & $1.63 \pm 2.12 \times 10^{-1}$ & $1.90 \times 10^{1} \pm 3.61$ & $2.74 \pm 6.02 \times 10^{-1}$ & $\mathbf{2.85 \times 10^{-1} \pm 4.42 \times 10^{-2}}$ & $7.30 \times 10^{-1} \pm 3.31 \times 10^{-1}$ \\
				\hline
				& 20 & $8.66 \times 10^{-3} \pm 5.26 \times 10^{-3}$ & $1.53 \times 10^{-8} \pm 4.57 \times 10^{-9}$ & $1.72 \times 10^{-9} \pm 6.92 \times 10^{-10}$ & $9.91 \times 10^{-8} \pm 7.70 \times 10^{-8}$ & $\mathbf{2.19 \times 10^{-10} \pm 1.29 \times 10^{-10}}$ \\
				$F_{12}$& 30 & $3.15 \times 10^{-2} \pm 1.41 \times 10^{-2}$ & $3.42 \times 10^{-8} \pm 9.83 \times 10^{-9}$ & $3.83 \times 10^{-9} \pm 1.82 \times 10^{-9}$ & $2.51 \times 10^{-6} \pm 2.88 \times 10^{-6}$ & $\mathbf{4.04 \times 10^{-10} \pm 1.71 \times 10^{-10}}$ \\
				& 50 & $6.11 \times 10^{-2} \pm 2.19 \times 10^{-2}$ & $8.01 \times 10^{-8} \pm 2.73 \times 10^{-8}$ & $9.29 \times 10^{-9} \pm 2.72 \times 10^{-9}$ & $3.49 \times 10^{-5} \pm 4.01 \times 10^{-5}$ & $\mathbf{7.00 \times 10^{-10} \pm 1.88 \times 10^{-10}}$ \\
				& 100 & $2.13 \times 10^{-1} \pm 4.75 \times 10^{-2}$ & $2.28 \times 10^{-7} \pm 4.18 \times 10^{-8}$ & $4.18 \times 10^{-8} \pm 1.21 \times 10^{-8}$ & $3.19 \times 10^{-4} \pm 3.87 \times 10^{-4}$ & $\mathbf{1.36 \times 10^{-9} \pm 1.79 \times 10^{-10}}$ \\
				\hline
				& 20 & $1.13 \times 10^{-11} \pm 5.95 \times 10^{-12}$ & $2.41 \times 10^{3} \pm 6.47 \times 10^{2}$ & $6.64 \times 10^{2} \pm 4.27 \times 10^{2}$ & $\mathbf{1.82 \times 10^{-12} \pm 0}$ & $2.97 \times 10^{-12} \pm 1.12 \times 10^{-12}$ \\
				$F_{13}$& 30 & $3.79 \times 10^{-11} \pm 6.89 \times 10^{-11}$ & $3.14 \times 10^{3} \pm 7.71 \times 10^{2}$ & $8.64 \times 10^{2} \pm 6.67 \times 10^{2}$ & $\mathbf{1.82 \times 10^{-12} \pm 0}$ & $3.64 \times 10^{-12} \pm 1.51 \times 10^{-12}$ \\
				& 50 & $4.16 \times 10^{2} \pm 3.62 \times 10^{2}$ & $5.56 \times 10^{3} \pm 1.31 \times 10^{3}$ & $2.00 \times 10^{3} \pm 9.44 \times 10^{2}$ & $\mathbf{2.55 \times 10^{-11} \pm 0}$ & $3.61 \times 10^{-11} \pm 7.79 \times 10^{-12}$ \\
				& 100 & $6.39 \times 10^{2} \pm 5.13 \times 10^{2}$ & $1.07 \times 10^{4} \pm 1.92 \times 10^{3}$ & $4.65 \times 10^{3} \pm 1.91 \times 10^{3}$ & $\mathbf{1.51 \times 10^{-10} \pm 1.54 \times 10^{-11}}$ & $1.72 \times 10^{-10} \pm 1.69 \times 10^{-11}$ \\
				\hline
				& 20 & $2.18 \times 10^{-13} \pm 8.21 \times 10^{-14}$ & $2.33 \times 10^{-18} \pm 4.63 \times 10^{-18}$ & $8.32 \times 10^{-22} \pm 3.84 \times 10^{-22}$ & $4.10 \times 10^{-20} \pm 2.42 \times 10^{-20}$ & $\mathbf{5.12 \times 10^{-24} \pm 4.43 \times 10^{-24}}$ \\
				$F_{14}$& 30 & $1.75 \times 10^{-13} \pm 5.71 \times 10^{-14}$ & $2.96 \times 10^{-19} \pm 2.96 \times 10^{-19}$ & $7.76 \times 10^{-22} \pm 2.30 \times 10^{-22}$ & $5.03 \times 10^{-20} \pm 2.42 \times 10^{-20}$ & $\mathbf{2.58 \times 10^{-23} \pm 1.31 \times 10^{-23}}$ \\
				& 50 & $1.09 \times 10^{-13} \pm 2.63 \times 10^{-14}$ & $2.06 \times 10^{-19} \pm 7.89 \times 10^{-20}$ & $5.98 \times 10^{-22} \pm 1.54 \times 10^{-22}$ & $4.48 \times 10^{-20} \pm 1.43 \times 10^{-20}$ & $\mathbf{9.38 \times 10^{-23} \pm 3.19 \times 10^{-23}}$ \\
				& 100 & $7.41 \times 10^{-14} \pm 1.27 \times 10^{-14}$ & $1.55 \times 10^{-19} \pm 1.19 \times 10^{-19}$ & $3.75 \times 10^{-22} \pm 7.30 \times 10^{-23}$ & $2.96 \times 10^{-20} \pm 7.23 \times 10^{-21}$ & $\mathbf{1.01 \times 10^{-22} \pm 2.71 \times 10^{-23}}$ \\
				\hline
				& 20 & $1.17 \times 10^{-11} \pm 3.31 \times 10^{-12}$ & $8.51 \times 10^{-18} \pm 3.20 \times 10^{-18}$ & $3.74 \times 10^{-20} \pm 1.64 \times 10^{-20}$ & $1.46 \times 10^{-18} \pm 8.73 \times 10^{-19}$ & $\mathbf{5.55 \times 10^{-23} \pm 4.91 \times 10^{-23}}$ \\
				$F_{15}$& 30 & $8.80 \times 10^{-12} \pm 2.71 \times 10^{-12}$ & $7.08 \times 10^{-18} \pm 2.59 \times 10^{-18}$ & $3.28 \times 10^{-20} \pm 1.09 \times 10^{-20}$ & $1.78 \times 10^{-18} \pm 8.42 \times 10^{-19}$ & $\mathbf{1.81 \times 10^{-22} \pm 1.30 \times 10^{-22}}$ \\
				& 50 & $7.55 \times 10^{-12} \pm 1.77 \times 10^{-12}$ & $8.89 \times 10^{-18} \pm 2.34 \times 10^{-18}$ & $3.14 \times 10^{-20} \pm 8.39 \times 10^{-21}$ & $2.39 \times 10^{-18} \pm 7.91 \times 10^{-19}$ & $\mathbf{2.29 \times 10^{-21} \pm 7.71 \times 10^{-22}}$ \\
				& 100 & $6.85 \times 10^{-12} \pm 1.18 \times 10^{-12}$ & $8.41 \times 10^{-18} \pm 1.48 \times 10^{-18}$ & $3.19 \times 10^{-20} \pm 6.77 \times 10^{-21}$ & $2.47 \times 10^{-18} \pm 5.65 \times 10^{-19}$ & $\mathbf{6.92 \times 10^{-21} \pm 1.84 \times 10^{-21}}$ \\
				\bottomrule[1pt]
				\hline
			\end{tabular}
		}%
	}
\end{table*}

\begin{table*}[!htbp]
    \centering
    \caption{Wilcoxon Rank-Sum Test Results of STA Variants}
    \label{tab:ranksum_STA_variants}
    \begin{threeparttable}
	    \renewcommand{\arraystretch}{1.2}
	    \begin{tabular}{lccc ccc ccc ccc}
	        \toprule
	        Algorithm & \multicolumn{3}{c}{20D} & \multicolumn{3}{c}{30D} & \multicolumn{3}{c}{50D} & \multicolumn{3}{c}{100D} \\
	        \cmidrule(lr){2-4} \cmidrule(lr){5-7} \cmidrule(lr){8-10} \cmidrule(lr){11-13}
	        & \#\texttt{+} & \#\texttt{=} & \#\texttt{-}
	        & \#\texttt{+} & \#\texttt{=} & \#\texttt{-}
	        & \#\texttt{+} & \#\texttt{=} & \#\texttt{-}
	        & \#\texttt{+} & \#\texttt{=} & \#\texttt{-} \\
	        \midrule
	        STA                 &  0 & 0 & 15 & 0 & 0 & 15 & 0 & 0 & 15 &  0 & 0 & 15\\
	        DaSTA               &  1 & 1 & 13 & 2 & 0 & 13 & 1 & 1 & 13 &  2 & 0 & 13\\
	        POSTA               &  7 & 2 & 6 & 9 & 0 & 6 & 8 & 0 & 7 &  7 & 0 & 8\\
	        EXSTA               &  14 & 0 & 1 & 11 & 2 & 2 & 12 & 2 & 1 &  12 & 0 & 3\\
	        \bottomrule
	    \end{tabular}
	    \begin{tablenotes}
	    	\footnotesize
	    	\item 	Note: \texttt{\#+}, \texttt{\#=}, and \texttt{\#-} represent the count of instances where one algorithm performs significantly better than, are statistically equivalent to, and are significantly worse than its counterparts, respectively.
	    \end{tablenotes}
    \end{threeparttable}
\end{table*}

\begin{figure*}[!t]
	\centering
	
	\subfloat[]{\includegraphics[width=0.32\textwidth]{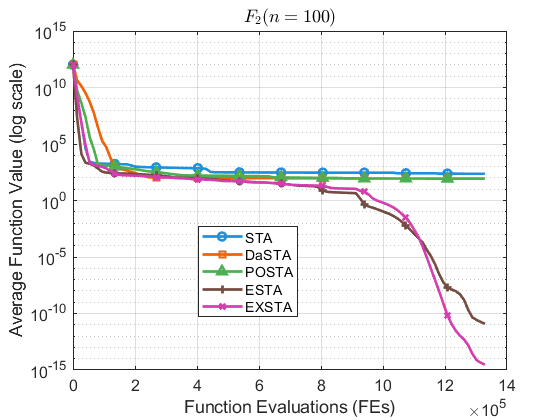}}
	\subfloat[]{\includegraphics[width=0.32\textwidth]{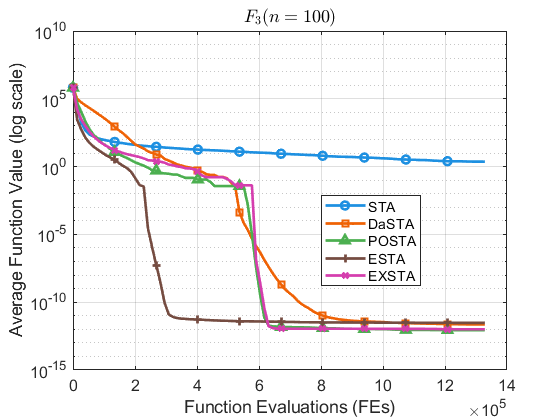}}
	\subfloat[]{\includegraphics[width=0.32\textwidth]{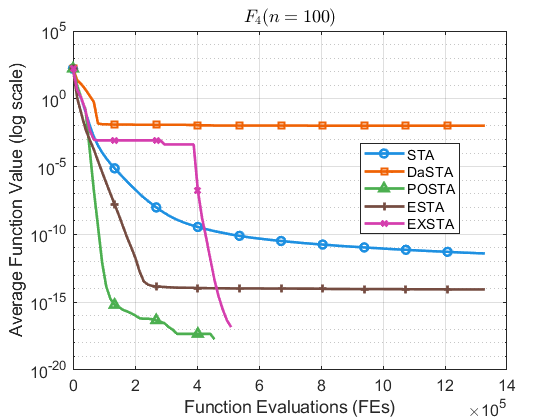}}\\[-1cm]

	\subfloat[]{\includegraphics[width=0.32\textwidth]{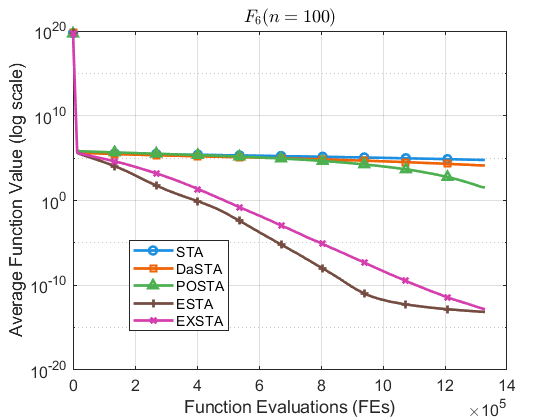}}
	\subfloat[]{\includegraphics[width=0.32\textwidth]{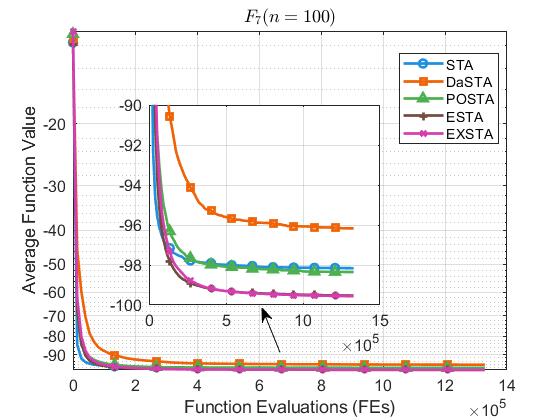}}
	\subfloat[]{\includegraphics[width=0.32\textwidth]{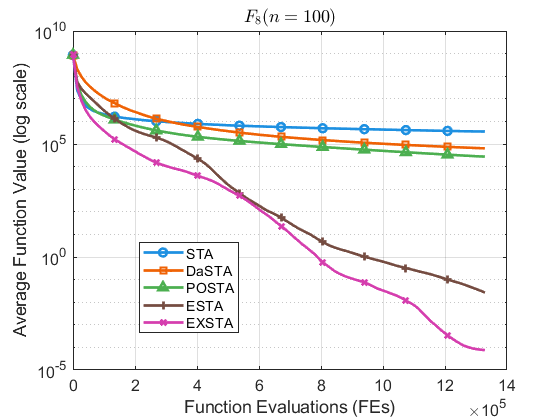}}

	\caption{Iterative curves of STA and its variants for 100-dimensional problems.}
	\label{fig:STA_variants}
\end{figure*}

\subsection{Comparison with other algorithms}
In this subsection, the proposed efficient STAs are compared with other algorithms under two termination criteria. The first criterion is the maximum number of stall iterations, used to evaluate whether the algorithm can achieve convergence to an optimal solution. The second criterion is the maximum number of function evaluations, used to assess the search capability of the algorithm.

\subsubsection{Terminate at maximum stall iterations}
In this part, the proposed method is compared with other algorithms under the termination condition of reaching the maximum number of stall iterations (maxStalls), which is set to maxStalls = 20. It is worth noting that the proposed method adopts its own designed termination condition, that is to say, in ESTA and EXSTA, maxStalls = 1. The average gradient norm (GradNorm) is reported in Table \ref{tab:gradnorm_maxstall}, with values greater than 0.1 highlighted in bold. It was found that with maxStalls = 1, both ESTA and EXSTA achieve convergence to optimality.  However, with maxStalls = 20, some algorithms are able to find the optimal solution for certain benchmark functions, but fail to do so for others. In other words, these algorithms cannot ensure convergence to optimality across all benchmark functions. A possible explanation is that some algorithms rely heavily on imitation learning or behavioral cloning, which lack the necessary mechanisms to ensure optimality, and they tend to stagnate at stable points \cite{zhou2025stagnation}.

\begin{table*}[!htbp]
	\centering
	\caption{Comparison of Average Gradient Norm (GradNorm) under the maxStalls Termination Criterion}
	\label{tab:gradnorm_maxstall}
	\tiny 
	\setlength{\tabcolsep}{3pt} 
	\begin{tabular}{l|c|c|c|c|c|c|c|c|c|c|c}
		\hline
		\toprule[1pt]
		Fcn & Dim & GL25 & CLPSO & CMA-ES & LSHADE & GWO & WOA & PGA & KLA & ESTA & EXSTA \\
		\hline
		& 20 & $\mathbf{8.74 \times 10^{1}}$ & $\mathbf{6.49 \times 10^{1}}$ & $2.04 \times 10^{-8}$ & $9.36 \times 10^{-9}$ & $\mathbf{3.12 \times 10^{1}}$ & $\mathbf{2.46}$ & $2.91 \times 10^{-8}$ & $2.37 \times 10^{-8}$ & $8.73 \times 10^{-8}$ & $4.38 \times 10^{-9}$  \\
		$F_{1}$& 30 & $\mathbf{1.87 \times 10^{2}}$ & $\mathbf{1.54 \times 10^{2}}$ & $\mathbf{9.06 \times 10^{1}}$ & $1.05 \times 10^{-8}$ & $\mathbf{6.84 \times 10^{1}}$ & $\mathbf{5.14}$ & $4.47 \times 10^{-8}$ & $3.62 \times 10^{-8}$ & $1.14 \times 10^{-7}$ & $7.49 \times 10^{-9}$  \\
		& 50 & $\mathbf{4.04 \times 10^{2}}$ & $\mathbf{4.50 \times 10^{2}}$ & $\mathbf{1.93 \times 10^{2}}$ & $1.15 \times 10^{-8}$ & $\mathbf{1.90 \times 10^{2}}$ & $\mathbf{1.48 \times 10^{1}}$ & $7.23 \times 10^{-8}$ & $6.28 \times 10^{-8}$ & $1.68 \times 10^{-7}$ & $1.47 \times 10^{-8}$  \\
		& 100 & $\mathbf{1.32 \times 10^{3}}$ & $\mathbf{1.45 \times 10^{3}}$ & $\mathbf{5.48 \times 10^{2}}$ & $1.76 \times 10^{-8}$ & $\mathbf{6.80 \times 10^{2}}$ & $\mathbf{6.17 \times 10^{1}}$ & $\mathbf{7.61 \times 10^{1}}$ & $1.41 \times 10^{-7}$ & $2.83 \times 10^{-7}$ & $2.85 \times 10^{-8}$  \\
		\hline
		& 20 & $\mathbf{6.44 \times 10^{6}}$ & $\mathbf{3.15 \times 10^{6}}$ & $4.37 \times 10^{-7}$ & $3.40 \times 10^{-7}$ & $\mathbf{1.65 \times 10^{5}}$ & $\mathbf{1.66 \times 10^{2}}$ & $\mathbf{2.83}$ & $\mathbf{6.31 \times 10^{-1}}$ & $4.10 \times 10^{-6}$ & $7.71 \times 10^{-7}$  \\
		$F_{2}$& 30 & $\mathbf{2.65 \times 10^{7}}$ & $\mathbf{2.65 \times 10^{7}}$ & $\mathbf{1.18 \times 10^{2}}$ & $3.84 \times 10^{-7}$ & $\mathbf{1.69 \times 10^{6}}$ & $\mathbf{1.02 \times 10^{3}}$ & $\mathbf{4.47}$ & $1.64 \times 10^{-2}$ & $4.56 \times 10^{-6}$ & $1.00 \times 10^{-6}$  \\
		& 50 & $\mathbf{2.40 \times 10^{8}}$ & $\mathbf{1.94 \times 10^{8}}$ & $1.18 \times 10^{-6}$ & $4.74 \times 10^{-7}$ & $\mathbf{2.50 \times 10^{7}}$ & $\mathbf{9.88 \times 10^{3}}$ & $\mathbf{5.07}$ & $1.92 \times 10^{-6}$ & $7.87 \times 10^{-6}$ & $1.45 \times 10^{-6}$  \\
		& 100 & $\mathbf{3.37 \times 10^{9}}$ & $\mathbf{4.64 \times 10^{9}}$ & $\mathbf{1.44 \times 10^{8}}$ & $9.80 \times 10^{-7}$ & $\mathbf{5.85 \times 10^{8}}$ & $\mathbf{4.39 \times 10^{5}}$ & $\mathbf{6.04 \times 10^{7}}$ & $2.76 \times 10^{-6}$ & $1.70 \times 10^{-5}$ & $3.10 \times 10^{-6}$  \\
		\hline
		& 20 & $\mathbf{2.02 \times 10^{2}}$ & $\mathbf{1.88 \times 10^{2}}$ & $\mathbf{1.87 \times 10^{2}}$ & $\mathbf{1.66 \times 10^{2}}$ & $\mathbf{1.99 \times 10^{2}}$ & $\mathbf{1.30 \times 10^{2}}$ & $\mathbf{2.11}$ & $\mathbf{1.85 \times 10^{2}}$ & $1.64 \times 10^{-5}$ & $9.74 \times 10^{-6}$  \\
		$F_{3}$& 30 & $\mathbf{3.40 \times 10^{2}}$ & $\mathbf{3.09 \times 10^{2}}$ & $\mathbf{2.37 \times 10^{2}}$ & $\mathbf{2.21 \times 10^{2}}$ & $\mathbf{2.42 \times 10^{2}}$ & $\mathbf{2.11 \times 10^{2}}$ & $\mathbf{2.22 \times 10^{2}}$ & $\mathbf{2.28 \times 10^{2}}$ & $3.18 \times 10^{-5}$ & $1.68 \times 10^{-5}$  \\
		& 50 & $\mathbf{4.93 \times 10^{2}}$ & $\mathbf{5.59 \times 10^{2}}$ & $\mathbf{3.47 \times 10^{2}}$ & $\mathbf{2.98 \times 10^{2}}$ & $\mathbf{3.59 \times 10^{2}}$ & $\mathbf{3.24 \times 10^{2}}$ & $\mathbf{1.51 \times 10^{2}}$ & $\mathbf{3.05 \times 10^{2}}$ & $5.00 \times 10^{-5}$ & $2.51 \times 10^{-5}$  \\
		& 100 & $\mathbf{1.32 \times 10^{3}}$ & $\mathbf{1.49 \times 10^{3}}$ & $\mathbf{6.96 \times 10^{2}}$ & $\mathbf{4.30 \times 10^{2}}$ & $\mathbf{8.05 \times 10^{2}}$ & $\mathbf{5.08 \times 10^{2}}$ & $\mathbf{4.16 \times 10^{2}}$ & $\mathbf{4.35 \times 10^{2}}$ & $1.40 \times 10^{-4}$ & $6.11 \times 10^{-5}$  \\
		\hline
		& 20 & $1.79 \times 10^{-2}$ & $1.72 \times 10^{-2}$ & $5.24 \times 10^{-3}$ & $1.07 \times 10^{-8}$ & $7.37 \times 10^{-3}$ & $\mathbf{1.03 \times 10^{-1}}$ & $1.11 \times 10^{-8}$ & $1.08 \times 10^{-8}$ & $2.88 \times 10^{-8}$ & $1.19 \times 10^{-8}$  \\
		$F_{4}$& 30 & $4.34 \times 10^{-2}$ & $4.22 \times 10^{-2}$ & $1.17 \times 10^{-2}$ & $1.02 \times 10^{-8}$ & $1.70 \times 10^{-2}$ & $\mathbf{1.59 \times 10^{-1}}$ & $1.49 \times 10^{-8}$ & $1.28 \times 10^{-8}$ & $3.80 \times 10^{-8}$ & $1.37 \times 10^{-8}$  \\
		& 50 & $\mathbf{1.04 \times 10^{-1}}$ & $\mathbf{1.06 \times 10^{-1}}$ & $4.37 \times 10^{-2}$ & $9.90 \times 10^{-9}$ & $4.78 \times 10^{-2}$ & $\mathbf{1.56 \times 10^{-1}}$ & $\mathbf{1.75 \times 10^{-1}}$ & $1.57 \times 10^{-8}$ & $4.35 \times 10^{-8}$ & $1.41 \times 10^{-8}$  \\
		& 100 & $\mathbf{3.39 \times 10^{-1}}$ & $\mathbf{3.56 \times 10^{-1}}$ & $\mathbf{1.44 \times 10^{-1}}$ & $1.02 \times 10^{-8}$ & $\mathbf{1.74 \times 10^{-1}}$ & $1.41 \times 10^{-2}$ & $\mathbf{1.42 \times 10^{-1}}$ & $2.48 \times 10^{-8}$ & $7.06 \times 10^{-8}$ & $1.61 \times 10^{-8}$  \\
		\hline
		& 20 & $\mathbf{1.01}$ & $\mathbf{1.12}$ & $\mathbf{1.37}$ & $6.45 \times 10^{-10}$ & $\mathbf{1.17}$ & $\mathbf{1.45}$ & $1.05 \times 10^{-9}$ & $9.90 \times 10^{-10}$ & $2.84 \times 10^{-4}$ & $1.25 \times 10^{-7}$  \\
		$F_{5}$& 30 & $\mathbf{8.89 \times 10^{-1}}$ & $\mathbf{1.17}$ & $\mathbf{1.15}$ & $1.25 \times 10^{-9}$ & $\mathbf{9.33 \times 10^{-1}}$ & $\mathbf{1.27}$ & $1.69 \times 10^{-9}$ & $2.86 \times 10^{-3}$ & $2.07 \times 10^{-4}$ & $1.35 \times 10^{-6}$  \\
		& 50 & $\mathbf{8.88 \times 10^{-1}}$ & $\mathbf{8.30 \times 10^{-1}}$ & $\mathbf{8.36 \times 10^{-1}}$ & $\mathbf{9.90 \times 10^{-1}}$ & $\mathbf{8.84 \times 10^{-1}}$ & $2.58 \times 10^{-3}$ & $\mathbf{1.32 \times 10^{-1}}$ & $2.88 \times 10^{-5}$ & $4.50 \times 10^{-4}$ & $5.50 \times 10^{-6}$  \\
		& 100 & $\mathbf{5.47 \times 10^{-1}}$ & $\mathbf{5.20 \times 10^{-1}}$ & $\mathbf{5.48 \times 10^{-1}}$ & $\mathbf{6.42 \times 10^{-1}}$ & $\mathbf{5.85 \times 10^{-1}}$ & $3.55 \times 10^{-6}$ & $3.40 \times 10^{-6}$ & $0.00$ & $2.80 \times 10^{-4}$ & $3.96 \times 10^{-7}$  \\
		\hline
		& 20 & $\mathbf{3.28 \times 10^{2}}$ & $\mathbf{5.32 \times 10^{3}}$ & $\mathbf{1.32 \times 10^{3}}$ & $7.96 \times 10^{-8}$ & $\mathbf{7.03 \times 10^{2}}$ & $\mathbf{1.96 \times 10^{2}}$ & $\mathbf{5.76 \times 10^{2}}$ & $\mathbf{8.15 \times 10^{2}}$ & $2.98 \times 10^{-7}$ & $3.62 \times 10^{-8}$  \\
		$F_{6}$& 30 & $\mathbf{4.16 \times 10^{7}}$ & $\mathbf{1.40 \times 10^{10}}$ & $\mathbf{8.30 \times 10^{3}}$ & $1.67 \times 10^{-7}$ & $\mathbf{3.45 \times 10^{3}}$ & $\mathbf{3.34 \times 10^{2}}$ & $\mathbf{6.24 \times 10^{3}}$ & $\mathbf{5.40 \times 10^{3}}$ & $6.72 \times 10^{-7}$ & $1.09 \times 10^{-7}$  \\
		& 50 & $\mathbf{1.86 \times 10^{12}}$ & $\mathbf{2.22 \times 10^{12}}$ & $\mathbf{3.69 \times 10^{4}}$ & $\mathbf{3.68 \times 10^{3}}$ & $\mathbf{1.25 \times 10^{5}}$ & $\mathbf{6.39 \times 10^{2}}$ & $\mathbf{1.35 \times 10^{4}}$ & $\mathbf{3.12 \times 10^{4}}$ & $8.01 \times 10^{-7}$ & $2.28 \times 10^{-7}$  \\
		& 100 & $\mathbf{6.78 \times 10^{16}}$ & $\mathbf{1.11 \times 10^{17}}$ & $\mathbf{1.88 \times 10^{6}}$ & $\mathbf{1.22 \times 10^{5}}$ & $\mathbf{3.30 \times 10^{5}}$ & $\mathbf{1.11 \times 10^{4}}$ & $\mathbf{8.41 \times 10^{5}}$ & $\mathbf{3.45 \times 10^{5}}$ & $9.37 \times 10^{-6}$ & $1.25 \times 10^{-6}$  \\
		\hline
		& 20 & $\mathbf{5.00 \times 10^{1}}$ & $\mathbf{7.00 \times 10^{1}}$ & $\mathbf{6.73 \times 10^{1}}$ & $\mathbf{7.52 \times 10^{1}}$ & $\mathbf{6.19 \times 10^{1}}$ & $\mathbf{4.15 \times 10^{1}}$ & $\mathbf{4.48 \times 10^{1}}$ & $\mathbf{6.27 \times 10^{1}}$ & $2.15 \times 10^{-5}$ & $9.92 \times 10^{-6}$  \\
		$F_{7}$& 30 & $\mathbf{1.27 \times 10^{2}}$ & $\mathbf{9.66 \times 10^{1}}$ & $\mathbf{1.15 \times 10^{2}}$ & $\mathbf{1.27 \times 10^{2}}$ & $\mathbf{1.10 \times 10^{2}}$ & $\mathbf{9.18 \times 10^{1}}$ & $\mathbf{1.25 \times 10^{2}}$ & $\mathbf{1.04 \times 10^{2}}$ & $5.70 \times 10^{-5}$ & $2.45 \times 10^{-5}$  \\
		& 50 & $\mathbf{1.80 \times 10^{2}}$ & $\mathbf{2.50 \times 10^{2}}$ & $\mathbf{2.39 \times 10^{2}}$ & $\mathbf{2.73 \times 10^{2}}$ & $\mathbf{2.19 \times 10^{2}}$ & $\mathbf{2.03 \times 10^{2}}$ & $\mathbf{2.40 \times 10^{2}}$ & $\mathbf{2.39 \times 10^{2}}$ & $1.99 \times 10^{-4}$ & $7.77 \times 10^{-5}$  \\
		& 100 & $\mathbf{5.67 \times 10^{2}}$ & $\mathbf{5.98 \times 10^{2}}$ & $\mathbf{6.38 \times 10^{2}}$ & $\mathbf{7.21 \times 10^{2}}$ & $\mathbf{5.42 \times 10^{2}}$ & $\mathbf{5.74 \times 10^{2}}$ & $\mathbf{5.97 \times 10^{2}}$ & $\mathbf{6.48 \times 10^{2}}$ & $1.30 \times 10^{-3}$ & $4.28 \times 10^{-4}$  \\
		\hline
		& 20 & $\mathbf{6.65 \times 10^{2}}$ & $\mathbf{4.65 \times 10^{2}}$ & $1.41 \times 10^{-5}$ & $1.23 \times 10^{-5}$ & $\mathbf{1.68 \times 10^{2}}$ & $\mathbf{8.81}$ & $9.50 \times 10^{-3}$ & $1.82 \times 10^{-5}$ & $3.94 \times 10^{-5}$ & $2.16 \times 10^{-5}$  \\
		$F_{8}$& 30 & $\mathbf{2.16 \times 10^{3}}$ & $\mathbf{1.97 \times 10^{3}}$ & $4.43 \times 10^{-5}$ & $3.90 \times 10^{-5}$ & $\mathbf{2.97 \times 10^{2}}$ & $\mathbf{1.22 \times 10^{1}}$ & $\mathbf{1.48 \times 10^{-1}}$ & $6.63 \times 10^{-5}$ & $1.56 \times 10^{-4}$ & $7.17 \times 10^{-5}$  \\
		& 50 & $\mathbf{8.51 \times 10^{3}}$ & $\mathbf{8.00 \times 10^{3}}$ & $5.76 \times 10^{-4}$ & $5.44 \times 10^{-4}$ & $\mathbf{3.51 \times 10^{1}}$ & $\mathbf{1.47 \times 10^{1}}$ & $\mathbf{1.67}$ & $2.55 \times 10^{-2}$ & $1.27 \times 10^{-3}$ & $6.75 \times 10^{-4}$  \\
		& 100 & $\mathbf{5.03 \times 10^{4}}$ & $\mathbf{5.72 \times 10^{4}}$ & $2.46 \times 10^{-2}$ & $2.28 \times 10^{-2}$ & $\mathbf{3.43 \times 10^{1}}$ & $\mathbf{2.01 \times 10^{1}}$ & $\mathbf{1.29 \times 10^{2}}$ & $\mathbf{9.64}$ & $2.52 \times 10^{-2}$ & $2.51 \times 10^{-2}$  \\
		\hline
		& 20 & $\mathbf{5.36 \times 10^{3}}$ & $\mathbf{3.88 \times 10^{3}}$ & $\mathbf{2.49 \times 10^{2}}$ & $4.01 \times 10^{-8}$ & $\mathbf{1.32 \times 10^{2}}$ & $\mathbf{9.06 \times 10^{1}}$ & $\mathbf{2.27 \times 10^{-1}}$ & $1.70 \times 10^{-7}$ & $1.71 \times 10^{-7}$ & $3.35 \times 10^{-8}$  \\
		$F_{9}$& 30 & $\mathbf{3.46 \times 10^{4}}$ & $\mathbf{4.24 \times 10^{4}}$ & $\mathbf{1.10 \times 10^{3}}$ & $5.53 \times 10^{-8}$ & $\mathbf{3.22 \times 10^{2}}$ & $\mathbf{3.29 \times 10^{2}}$ & $\mathbf{3.13 \times 10^{1}}$ & $\mathbf{4.95 \times 10^{-1}}$ & $3.01 \times 10^{-7}$ & $5.10 \times 10^{-8}$  \\
		& 50 & $\mathbf{2.46 \times 10^{5}}$ & $\mathbf{2.74 \times 10^{5}}$ & $\mathbf{8.42 \times 10^{3}}$ & $1.19 \times 10^{-7}$ & $\mathbf{1.94 \times 10^{3}}$ & $\mathbf{1.43 \times 10^{3}}$ & $\mathbf{5.30 \times 10^{2}}$ & $\mathbf{9.73 \times 10^{1}}$ & $5.80 \times 10^{-7}$ & $1.10 \times 10^{-7}$  \\
		& 100 & $\mathbf{3.25 \times 10^{6}}$ & $\mathbf{3.81 \times 10^{6}}$ & $\mathbf{1.42 \times 10^{5}}$ & $3.26 \times 10^{-7}$ & $\mathbf{1.68 \times 10^{4}}$ & $\mathbf{1.29 \times 10^{4}}$ & $\mathbf{4.97 \times 10^{3}}$ & $\mathbf{1.41 \times 10^{3}}$ & $1.25 \times 10^{-6}$ & $3.96 \times 10^{-7}$  \\
		\hline
		& 20 & $\mathbf{3.76 \times 10^{4}}$ & $\mathbf{3.59 \times 10^{4}}$ & $5.06 \times 10^{-8}$ & $2.42 \times 10^{-8}$ & $\mathbf{1.90 \times 10^{3}}$ & $\mathbf{9.17}$ & $1.04 \times 10^{-7}$ & $1.03 \times 10^{-7}$ & $2.02 \times 10^{-7}$ & $1.77 \times 10^{-8}$  \\
		$F_{10}$& 30 & $\mathbf{2.59 \times 10^{5}}$ & $\mathbf{2.64 \times 10^{5}}$ & $6.43 \times 10^{-8}$ & $2.70 \times 10^{-8}$ & $\mathbf{1.90 \times 10^{4}}$ & $\mathbf{6.07 \times 10^{1}}$ & $1.41 \times 10^{-7}$ & $1.61 \times 10^{-7}$ & $2.75 \times 10^{-7}$ & $2.96 \times 10^{-8}$  \\
		& 50 & $\mathbf{2.22 \times 10^{6}}$ & $\mathbf{2.73 \times 10^{6}}$ & $\mathbf{1.10 \times 10^{5}}$ & $2.94 \times 10^{-8}$ & $\mathbf{2.56 \times 10^{5}}$ & $\mathbf{3.26 \times 10^{2}}$ & $2.27 \times 10^{-7}$ & $2.23 \times 10^{-7}$ & $4.11 \times 10^{-7}$ & $4.94 \times 10^{-8}$  \\
		& 100 & $\mathbf{3.33 \times 10^{7}}$ & $\mathbf{4.35 \times 10^{7}}$ & $\mathbf{1.72 \times 10^{6}}$ & $4.71 \times 10^{-8}$ & $\mathbf{6.27 \times 10^{6}}$ & $\mathbf{5.01 \times 10^{3}}$ & $\mathbf{3.22 \times 10^{5}}$ & $4.40 \times 10^{-7}$ & $5.47 \times 10^{-7}$ & $9.44 \times 10^{-8}$  \\
		\bottomrule[1pt]
		\hline
	\end{tabular}
\end{table*}

\subsubsection{Terminate at maximum function evaluations}
This section evaluates the proposed effective STAs against other algorithms for the sensor network localization (SNL) problem, using a fixed budget of maximum function evaluations (maxFEs). Following the formulation in \cite{zhou2018dynamic}, the SNL problem is described as follows: Given $m$ anchor points $\bm a_1, \dots, \bm a_m \in \mathbb{R}^d$ ($d=2$ in this study) and a radio range $r_d$, let $N_x = \{(i,j): \|\bm x_i - \bm x_j\| = d_{ij} \leq r_d\}$ denote the set of measurable sensor-sensor distances, and $N_a = \{(i,k): \|\bm x_i - \bm a_k\| = \bar{d}_{ik} \leq r_d\}$ denote the set of measurable sensor-anchor distances. The goal is to determine the positions of $n$ distinct sensor points $\bm x_i \in \mathbb{R}^2$, for $i = 1, \dots, n$, that satisfy the distance constraints for all pairs in $N_x$ and $N_a$ as follows
\addtocounter{equation}{1}
\begin{align}\label{eqwsn}
\|\bm x_i - \bm x_j\|^2 \!=\! d_{ij}^2, \forall (i,j) \in N_x, \tag{\theequation a}\\
\|\bm x_i - \bm a_k\|^2 \!=\! \bar{d}_{ik}^2, \forall (i,k) \in N_a. \tag{\theequation b}
\end{align}

Employing the nonlinear least squares method, the SNL problem can be reformulated as the following nonconvex optimization problem:
\begin{equation}
\min \sum_{(i,j) \in N_x} (\|\bm x_i \!-\! \bm x_j\|^2 \!-\! d_{ij}^2)^2 \!+\! \sum_{(i,k) \in N_a} (\|\bm x_i \!-\! \bm a_k\|^2 \!-\! \bar{d}_{ik}^2)^2.
\end{equation}

In the experiments, the SNL problem with 100 sensors is considered. The maximum number of function evaluations (maxFEs) is set to $10^4 \times D$, where $D = 2n = 200$. The search domain is restricted to the interval $[0,1]$. Each algorithm is independently executed 30 times. The statistical results are summarized in Table \ref{tab:snlproblem}, while the best solution obtained by each algorithm is illustrated in Fig. \ref{fig:all_WSN}, where green circles denote the true sensor positions and red asterisks represent the estimated positions.

It is observed that CMA-ES, ESTA, and EXSTA consistently achieve the global best solution with high precision, but CMA-ES requires more computational time.
That is to say, the trade-off between computational cost and efficiency still exists. It is observed that ESTA occasionally incurs high computational time due to its parameter settings, while EXSTA's computation time is competitive with that of its competitors. Notably, EXSTA achieves the very high accuracy while requiring significantly less computational time. In contrast, GWO performs poorly on this problem.

\begin{table*}[!htbp]
	\centering
	\caption{Statistical Results of Different Algorithms on the SNL Problem}
	\label{tab:snlproblem}
	\tiny 
	\begin{tabular}{l|c|c|c|c|c|c|c|c|c|c}
		\hline
		\toprule[1pt]
		Index & GL25 & CLPSO & CMA-ES & LSHADE & GWO & WOA & PGA & KLA & ESTA & EXSTA \\
		\hline
		Best    & $2.26 \times 10^{-1}$ & $5.09 \times 10^{-4}$ & $1.48 \times 10^{-28}$ & $1.92 \times 10^{-8}$ & $2.25 \times 10^{1}$ & $4.30 \times 10^{-2}$ & $1.15 \times 10^{-2}$ & $6.41 \times 10^{-6}$ & $6.41 \times 10^{-15}$ & $1.87 \times 10^{-18}$ \\
		Mean    & $6.07 \times 10^{-1}$ & $7.30 \times 10^{-4}$ & $1.62 \times 10^{-28}$ & $1.28 \times 10^{-6}$ & $2.72 \times 10^{1}$ & $4.16 \times 10^{-1}$ & $1.68$ & $4.72 \times 10^{-5}$ & $4.58 \times 10^{-14}$ & $1.51 \times 10^{-17}$ \\
		Worst   & $1.07$ & $1.90 \times 10^{-3}$ & $1.79 \times 10^{-28}$ & $4.44 \times 10^{-6}$ & $3.36 \times 10^{1}$ & $9.39 \times 10^{-1}$ & $5.27$ & $1.60 \times 10^{-4}$ & $1.10 \times 10^{-13}$ & $4.11 \times 10^{-17}$ \\
		Time(s) & 82.6       & 75.1       & 440.3      & 86.2       & 78.4       & 65.4       & 68.6       &  87.9      & 84.5       & 68.3       \\
		\bottomrule[1pt]
		\hline
	\end{tabular}
\end{table*}

\begin{figure*}[!htbp]
	\centering
	
	\makebox[\textwidth][c]{%
		\subfloat[]{\includegraphics[width=0.32\textwidth]{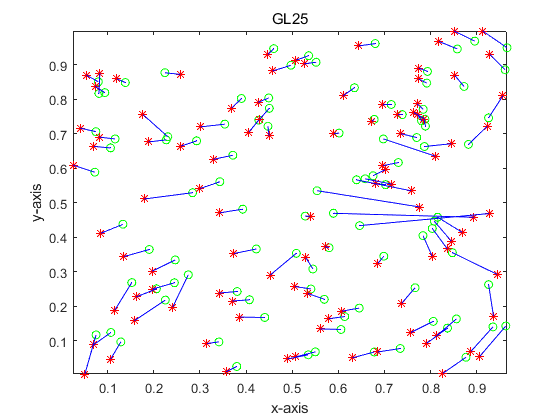}\label{fig:wsnGL25}}
		\subfloat[]{\includegraphics[width=0.32\textwidth]{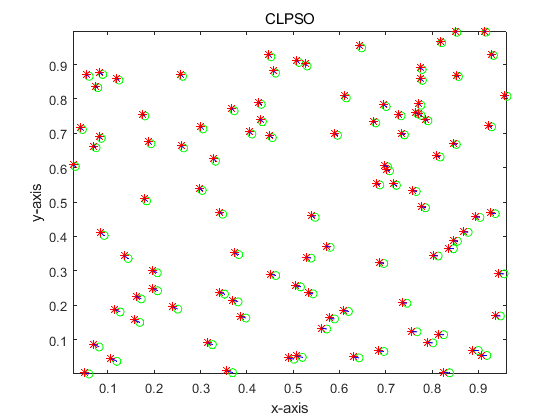}\label{fig:wsnCLPSO}}
		\subfloat[]{\includegraphics[width=0.32\textwidth]{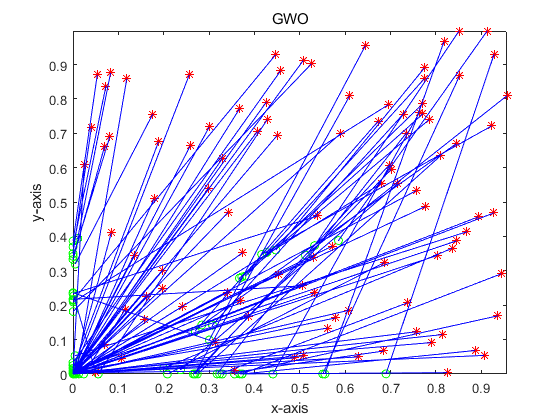}\label{fig:wsnGWO}}
	}

    \vspace{-1cm}
	
	\makebox[\textwidth][c]{%
		\subfloat[]{\includegraphics[width=0.32\textwidth]{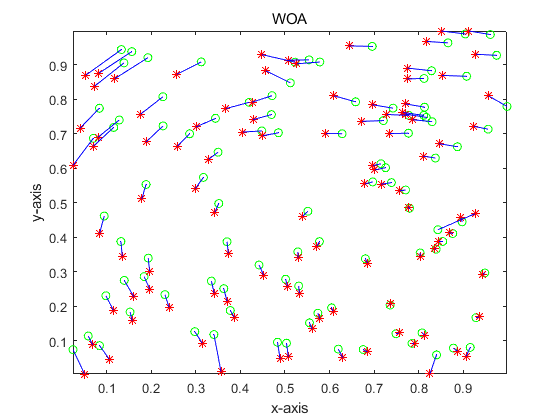}\label{fig:wsnWOA}}
		\subfloat[]{\includegraphics[width=0.32\textwidth]{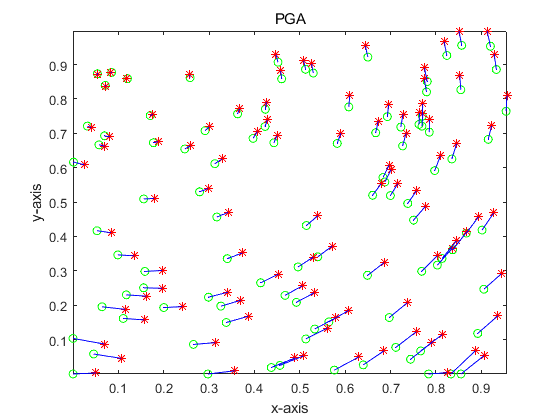}\label{fig:wsnPGA}}
		\subfloat[]{\includegraphics[width=0.32\textwidth]{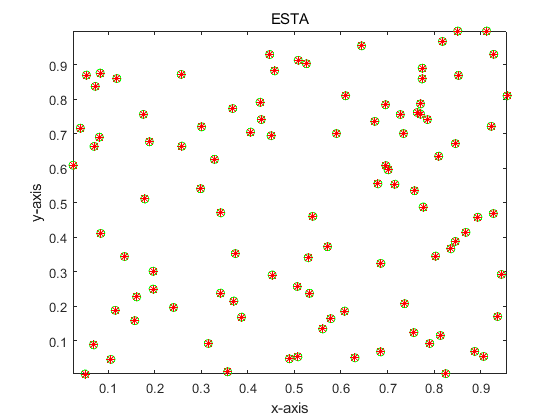}\label{fig:wsnESTA}}
	}

	\caption{Best results obtained by different algorithms for the SNL problem.}
	\label{fig:all_WSN}
\end{figure*}

\section{Conclusion}
This study proposes efficient STAs with guaranteed optimality for system engineering optimization. First, novel translation transformations based on predictive models are designed to generate more promising candidate solutions, thereby significantly accelerating convergence. Second, parameter control strategies are developed to enhance search efficiency. Finally, a dedicated termination criterion is designed to enable convergence to optimality. Comparative studies with both well-established and recently developed metaheuristics demonstrate the effectiveness and superiority of the proposed approach. Notably, even with a relatively small sample size (\emph{SE} = 30), the proposed ESTA and EXSTA are capable of achieving high-precision solutions close to the global optimum. It is also observed that ESTA may occasionally converge to local optima on low-dimensional Griewank functions, where near-global solutions are located far from the true global optimum, making it difficult to escape local minima.

It should be noted that the proposed ESTA occasionally yields locally optimal results on the low-dimensional Griewank function, as the near-global solution in the Griewank function is located far from the true global solution. In such cases, the proposed ESTA faces difficulty in escaping local minima.
Nevertheless, the proposed approach, along with the majority of existing studies on STA, primarily focuses on individual-based and sampling-based methods.
Considering that the population-based algorithms can explore different regions of the search space simultaneously, massively reducing the chance of getting trapped in a single local optimum, in the future, population-based STA will be further explored to enhance its global search capabilities. In the meanwhile, more real-world applications will be investigated to further validate the effectiveness and applicability of the proposed method.


\begin{appendices}
\section{The benchmark functions}
This appendix summarizes the benchmark functions adopted in the experiments, including their mathematical formulations and properties.

\begin{table*}[!htbp]
\caption{The Used Benchmark Functions}
\label{tab:ESTA_benchmarks}
\centering
\begin{tabular}{l|l|c|c|c}
\hline
\toprule[1pt]
\rowcolor{gray!20}
Name & Function & Range & $x^{*}$ & $F(x^{*})$ \\
\hline
shifted  Sphere      & $\begin{array}{l} F_1(\bm x) = f_1(\bm x- \bm s),\;\;\;\; \bm s = 1:n \\
	f_1(\bm x)= \displaystyle\sum_{i=1}^n x_{i}^{2} \end{array}$                              & [0, 2$n$] & [1,$\cdots$,$n$] & 0\\
\hline
\rowcolor{gray!20}
shifted Rosenbrock   & $\begin{array}{l} F_2(x) = f_2(x-s),\;\;\;\; s = 1:n \\
f_2(\bm x)= \displaystyle\sum_{i=1}^n (100(x_{i+1}-x_{i}^2)^2 + (x_{i}-1)^2) \end{array} $ & [0, 2$n$]   & [2,$\cdots$,$n+1$] & 0\\
\hline
shifted Rastrigin   & $\begin{array}{l} F_3(x) = f_3(x-s),\;\;\;\; s = 1:n \\
	 f_3(\bm x) =  \displaystyle\sum_{i=1}^n(x_{i}^{2}-10\cos(2 \pi x_{i})+10)\end{array}$    & [0, 2$n$]   & [1,$\cdots$,$n$] & 0\\
\hline
\rowcolor{gray!20}
shifted  Griewank    & $\begin{array}{l} F_4(x) = f_4(x-s),\;\;\;\; s = 1:n \\
	 f_4(\bm x) = \frac{1}{4000} \displaystyle \sum_{i=1}^n x_{i}^{2}-  \prod_{i=1}^{n} \cos(\frac{x_{i}}{\sqrt{i}}) + 1 \end{array}$ &
[0, 2$n$]   & [1,$\cdots$,$n$] & 0 \\
\hline
shifted Ackley     & $\begin{array}{l} F_5(x) = f_5(x-s),\;\;\;\; s = 1:n \\
	f_5(\bm x) = 20 + e - 20\exp\Big( -0.2  \scriptstyle{\sqrt{\frac{1}{n} \displaystyle \sum_{i=1}^n x_{i}^2}}\Big)-\exp\Big(\frac{1}{n} \displaystyle \sum_{i=1}^n \cos(2\pi x_{i})\Big)  \end{array}$ & [0,  2$n$]   & [1,$\cdots$,$n$] & 0\\
\hline
\rowcolor{gray!20}
shifted Zakharov    &  $\begin{array}{l} F_6(x) = f_6(x-s),\;\;\;\; s = 1:n \\
f_6(\bm x) = \displaystyle \sum_{i=1}^{n} x_i^2 + \Big( \displaystyle \sum_{i=1}^{n} 0.5 i x_i \Big)^2 + \Big( \displaystyle \sum_{i=1}^{n} 0.5 i x_i \Big)^4 \end{array} $ & [0,  2$n$]   & [1,$\cdots$,$n$] & 0 \\
\hline
Michalewicz & $F_7(\bm x) = -\displaystyle\sum_{i=1}^n \sin(x_{i}) \sin( \frac{i x_{i}^2}{\pi})^{20}$ & [0, $\pi$]   & unknown & unknown \\
\hline
\rowcolor{gray!20}
Trid        & $F_8(\bm x) = \displaystyle\sum_{i=1}^n (x_i - 1)^2 -\displaystyle\sum_{i=2}^n x_i x_{i-1} + \frac{n(n+4)(n-1)}{6}$ & $[0, n^2]$ & $x_i^{*} = i(n+1-i)$ & 0\\
\hline
shifted Schwefel 1.2 &  $\begin{array}{l} F_9(\bm x) = f_9(\bm x- \bm s),\;\;\;\; \bm s = 1:n \\
f_9(\bm x)= \displaystyle\sum_{i=1}^{n} \Big(\sum_{j=1}^{i} x_j \Big)^2 \end{array}$
& [0, 2$n$]   & [1,$\cdots$,$n$] & 0\\
\hline
\rowcolor{gray!20}
shifted Schwefel 2.4 & $\begin{array}{l} F_{10}(\bm x) = f_{10}(\bm x- \bm s),\;\;\;\; \bm s = 1:n \\
f_{10}(\bm x) =  \displaystyle\sum_{i=1}^{n} (x_i - 1)^2 + (x_1 - x_i^2)^2 \end{array}$ & [0, 2$n$]   & [2,$\cdots$,$n+1$] & 0 \\
\hline
shifted Schwefel 2.21 & $\begin{array}{l} F_{11}(\bm x) = f_{11}(\bm x- \bm s),\;\;\;\; \bm s = 1:n \\
f_{11}(\bm x) = \max \{ |x_i|, 1 \leq i \leq n \} \end{array} $ & [0, 2$n$]   & [1,$\cdots$,$n$] & 0 \\
\hline
\rowcolor{gray!20}
shifted Schwefel 2.22 & $\begin{array}{l} F_{12}(\bm x) = f_{12}(\bm x- \bm s),\;\;\;\; \bm s = 1:n \\
f_{12}(\bm x)  =  \displaystyle\sum_{i=1}^n |x_i| + \prod_{i=1}^n |x_i| \end{array} $ & [0, 2$n$]   & [1,$\cdots$,$n$] & 0 \\
\hline
Schwefel 2.26 & $F_{13}(\bm x) =   \displaystyle\sum_{i=1}^n -x_i \sin(\sqrt{|x_i|}) + 418.9829n$ & [-500, 500]   & [420.9687,$\cdots$] & 0 \\
\hline
\rowcolor{gray!20}
shifted Penalized1   & $\begin{array}{l} F_{14}(\bm x) = f_{14}(\bm x- \bm s),\;\;\;\; \bm s = 1:n \\
\begin{array}{l}  f_{14}(\bm x) = \frac{\pi}{n} \biggl\{ 10 \sin (\pi y_1) + \displaystyle\sum_{i=1}^{n-1} (y_i - 1)^2 \left[ 1 + 10 \sin^2 (\pi y_{i+1}) \right]  \\
  + (y_n - 1)^2 \biggr\} + \displaystyle \sum_{i=1}^{n} u(x_i, 10, 100, 4), y_i = 1 + \frac{x_i + 1}{4}\\
 u(x_i, a, k, m) =
\begin{cases}
k(x_i - a)^m, & x_i > a \\
0, & -a \leq x_i \leq a \\
k(-x_i - a)^m, & x_i < -a
\end{cases}
 \end{array} \end{array}$ &   [0, 2$n$]   & [0,$\cdots$,$n-1$] & 0 \\
\hline
shifted Penalized2   & $\begin{array}{l} F_{15}(\bm x) = f_{15}(\bm x- \bm s),\;\;\;\; \bm s = 1:n \\
\begin{array}{l}  f_{15}(\bm x) = 0.1 \left\{ \sin^2(3\pi x_1) + \displaystyle \sum_{i=1}^{n-1} (x_i - 1)^2 \left[1 + \sin^2(3\pi x_{i+1})\right] \right\} + \\
(x_n - 1)^2 \left[1 + \sin^2(2\pi x_n)\right] + \displaystyle  \sum_{i=1}^n u(x_i, 5, 100, 4) \end{array} \end{array}$ &   [0, 2$n$]   & [2,$\cdots$,$n+1$] & 0 \\
\bottomrule[1pt]
\hline
\end{tabular}
\end{table*}

\end{appendices}

\bibliographystyle{IEEEtran}
\bibliography{zxj}

\vfill

\end{document}